\documentclass{amsart}

\usepackage{hyperref,color,amsthm,amssymb,graphicx}
\usepackage{thm-restate}
\usepackage{enumitem}
\newlist{axiomlist}{enumerate}{1}
\setlist[axiomlist]{label=(S\arabic*),ref=(S\arabic*)}

\newtheorem{theorem}{Theorem}[section]
\newtheorem{lemma}[theorem]{Lemma}
\newtheorem{corollary}[theorem]{Corollary}
\newtheorem{proposition}[theorem]{Proposition}
\newtheorem{question}[theorem]{Question}

\newtheorem{claim}{Claim}
\newtheorem*{claim*}{Claim}
\newtheorem*{subclaim*}{Subclaim}

\theoremstyle{definition}
\newtheorem{remark}[theorem]{Remark}
\newtheorem{definition}[theorem]{Definition}
\newtheorem{example}[theorem]{Example}
\newtheorem{convention}[theorem]{Convention}

\newcommand{\acts}{\curvearrowright}
\newcommand{\mc}[1]{\mathcal{#1}}
\newcommand{\from}{\colon\thinspace}
\newcommand{\Isom}{\operatorname{Isom}}
\newcommand{\diam}{\operatorname{diam}}
\newcommand{\bR}{\mathbb{R}}
\newcommand{\bZ}{\mathbb{Z}}
\newcommand{\OutFn}{\operatorname{Out}(F_n)}
\newcommand{\Stab}{\operatorname{Stab}}
\newcommand{\te}{t^{-n}A_\emptyset}
\newcommand{\Cay}{\operatorname{Cay}}

\newcommand{\putinbox}[1]{\noindent\fbox{\parbox{\textwidth}{#1}}}

\begin{document}

\title[A/QI triples]{Acylindrically hyperbolic groups and their quasi-isometrically embedded subgroups}

\author[C.R. Abbott]{Carolyn R. Abbott}
\address{Department of Mathematics, Brandeis University, Waltham, MA 02453}
\email{carolynabbott@brandeis.edu}

\author[J.F. Manning]{Jason F. Manning}
\address{Department of Mathematics, 310 Malott Hall, Cornell University, Ithaca, NY 14853}
\email{jfmanning@math.cornell.edu}

\begin{abstract}
  We abstract the notion of an \emph{A/QI triple} from a number of examples in geometric group theory.  Such a triple $(G,X,H)$ consists of a group $G$ acting on a Gromov hyperbolic space $X$, acylindrically along a finitely generated subgroup $H$ which is quasi-isometrically embedded by the action.  Examples include strongly quasi-convex subgroups of relatively hyperbolic groups, convex cocompact subgroups of mapping class groups, many known convex cocompact subgroups of $\OutFn$, and groups generated by powers of independent loxodromic WPD elements of a group acting on a Gromov hyperbolic space.
  We initiate the study of intersection and combination properties of A/QI triples.
  Under the additional hypothesis that $G$ is finitely generated, we use a method of Sisto to show that $H$ is stable.  We apply theorems of Kapovich--Rafi and Dowdall--Taylor to analyze the Gromov boundary of an associated cone-off.  We close with some examples and questions.  
\end{abstract}
\maketitle

\section{Introduction}
There are two important threads in the study of group actions on hyperbolic spaces which we bring together in this paper.  One is the topic of quasi-convex subgroups of hyperbolic groups.  Gromov introduced these in~\cite{Gromov:HG} as the geometrically natural subgroups of a hyperbolic group, and they have been intensively studied since then.  Without pretending to give a full literature review, we mention in particular some previous work closely related to what we do here.  Gitik--Mitra--Rips--Sageev study finiteness and intersection properties of quasi-convex subgroups and show that they have finite height and width \cite{GMRS}.  Bowditch shows that cone-offs of hyperbolic spaces by families of quasi-convex subspaces are themselves hyperbolic \cite{Bowditch12}.
Other important early papers related to our current work are~\cite{Short91,Swenson01}.  

The other thread is that of acylindrical (and weakly properly discontinuous) actions of groups on hyperbolic spaces.  Acylindrical actions on \emph{trees} were studied first by Sela~\cite{Sela97}.  Bowditch formulated the acylindricity condition for actions on hyperbolic spaces in~\cite{Bowditch08}, showing in particular that the action of the mapping class group of a surface on the curve complex is acylindrical.  Osin in~\cite{Osin16} shows that the \emph{existence} of an acylindrical action of a group $G$ on some hyperbolic space is equivalent to the existence of a hyperbolically embedded subgroup, or of a weakly properly discontinuous element for the action of $G$ on some possibly different hyperbolic space, clarifying the existence of an extremely natural class of \emph{acylindrically hyperbolic groups}.  For such a group, Abbott--Balasubramanya--Osin explore the poset of acylindrical actions on hyperbolic spaces in~\cite{ABO}; this poset is further explored in~\cite{Abbott16,ABD,Balasubramanya,Val}.

In this paper we are interested in understanding actions of a group on hyperbolic spaces which are not necessarily acylindrical, but satisfy the weaker condition of being acylindrical along a finitely generated subgroup which is assumed to be quasi-isometrically embedded by the action. 
\begin{definition}\label{def:acyl}
Suppose $(X,d)$ is a Gromov hyperbolic geodesic metric space, and $Y\subseteq X$.  An isometric action $G\acts X$ is \emph{acylindrical along $Y$} if for every $\varepsilon \ge 0$, there are $R = R(\varepsilon, Y)> 0, M = M(\varepsilon, Y) > 0$ so that, whenever $x,y\in Y$ satisfy $d(x,y)\ge R$,
\begin{equation}\label{eq:acyldef}
  \#\{g\in G\mid d(gx,x)\le \varepsilon\mbox{ and }d(gy,y)\le \varepsilon\}\le M.
\end{equation} 
\end{definition}
While the constants $R, M$ are sometimes called the constants of acylindricity of the action, we reserve this terminology for the constants in an alternate characterization of acylindricity  which is easier to check in our context (see Definition~\ref{def:newacyl}).    If $G\acts X$ is acylindrical along $Y$, then it is acylindrical along $W$ for any $W$ which is finite Hausdorff distance from $Y$.  It is also acylindrical along any translate $gY$.
This notion interpolates between an acylindrical action and an action with a loxodromic WPD element (see Definition \ref{def:WPD}), in the following sense.
  The usual notion of an acylindrical action $G\acts X$ is one which is acylindrical along the entire space $X$, while an action $G\acts X$ with a loxodromic WPD element $g$ is acylindrical along any axis for $g$.

If $H<G$, we say that $G\acts X$ is \emph{acylindrical along $H$} if it is acylindrical along $Hx$ for some (equivalently any) $x\in X$.  
The main new definition of the current paper is the following:
\begin{definition}
  An \emph{A/QI triple} $(G,X,H)$ consists of a Gromov hyperbolic space $X$ and an action $G\acts X$ which is acylindrical along the finitely generated infinite subgroup $H$, and so that $x\mapsto hx$ gives a quasi-isometric embedding of $H$ into $X$.
  
  Whenever this final condition is satisfied, we say that $H$ is \emph{quasi-isometrically embedded by the action}.
\end{definition}

For example, if $G$ is hyperbolic, $H<G$ is quasi-convex, and $\Gamma$ is a Cayley graph of $G$, then $(G,\Gamma,H)$ is an A/QI triple.  More interesting motivating examples are convex cocompact subgroups of mapping class groups.
In this case,  if $H$ is a convex cocompact subgroup of a mapping class group $\operatorname{Mod}(\Sigma)$ of a non-exceptional surface $\Sigma$, and $\mathcal{C}$ is the curve complex of $\Sigma$, then $(\operatorname{Mod}(\Sigma),\mathcal{C},H)$ is an A/QI triple.  In this situation, the action of $G$ on $\mathcal C$ is a genuine acylindrical action.  Other motivating examples include certain convex cocompact subgroups $H$ of $\OutFn$, the outer automorphism group of the free group of rank $n$.  Indeed for many subgroups $H$, if $X$ is the free factor complex, the cyclic splitting complex, or the free splitting complex, then $(\OutFn, X, H)$ is an A/QI triple (see Example \ref{ex:outfn}).  None of these actions $\OutFn\acts X$ are known to be acylindrical, and in the case where $X$ is the free splitting complex, it is known that the action is not acylindrical.   That these examples are A/QI triples is an application of a general construction of A/QI triples from an action of a group on a hyperbolic space with a loxodromic WPD element (Corollary \ref{prop:WPDtoAQI}).  We use this construction to give additional examples of A/QI triples in the Cremona group and FC-type Artin groups (see Examples \ref{ex:cremona} and \ref{ex:artin}), neither of which have
known \emph{natural} acylindrical actions on hyperbolic spaces.

\begin{remark} \label{rmk:sisto}
In \cite{Sisto16},  Sisto gives a more flexible definition of an action being acylindrical along a subspace $Y$, omitting the number $M$ and replacing the inequality ``$\le M$'' in~\eqref{eq:acyldef} by ``$<\infty$''.  We will refer to an action satisfying this weaker condition as being \emph{weakly acylindrical along $Y$}.
In general an action can be weakly acylindrical along a subspace without being acylindrical along that subspace.  However if $X$ is hyperbolic, $Y$ is a quasi-convex subspace, and $d_X|_Y$ is a proper metric, then any action which is weakly acylindrical along $Y$ is acylindrical along $Y$.  In particular, if $Y=Hx$ and $H$ is quasi-isometrically embedded by the action, then the action is acylindrical along $Y$ if and only if it is weakly acylindrical along $Y$.
\end{remark}

In this paper, we study three aspects of A/QI triples $(G,X,H)$:  intersection and combination properties, stability of the subgroup $H$, and the boundary of a certain coned-off space of $X$ constructed from $(G,X,H)$. 

Motivated by the fact that the intersection of two quasi-convex subgroups of a hyperbolic group is quasi-convex, we first consider the intersection of A/QI triples.  We show that if $(G,X,H)$ and $(G,X,K)$ are A/QI triples, then $(G,X,H\cap K)$ is also an A/QI triple (Proposition \ref{prop:intersectedtriples}).   We also prove the following combination theorem for A/QI triples.  (Here we use the notation $\Lambda(H)$ for the \emph{limit set} of $H$; see Definition~\ref{def:limitset}.)

\begin{restatable}{theorem}{combination}\label{thm:combination}
  Let $G$ act on a hyperbolic space $X$, and suppose $H_1,\dots, H_k$ are subgroups of $G$ such that $(G,X,H_i)$ is an A/QI triple for each $i=1,\dots, k$ and $\Lambda(H_i)\cap \Lambda(H_j)=\emptyset$ for all $i\neq j$.  Then there exists a finite collection of elements $\mathcal S\subset \cup_{1=1}^k H_i\setminus\{1\}$ such that  for any family of subgroups $\mc{K}= \{K_i\}_{i = 1}^k$ with $K_i$ quasi-isometrically embedded in $H_i$ and satisfying
  $K_i\cap \mathcal S=\emptyset$ for all $i$, we have $K=\langle \bigcup\mc{K}\rangle\cong K_1\ast K_2\ast \cdots \ast K_k$ and $(G,X,K)$ is an A/QI triple.
\end{restatable}

Our first finiteness result is that for any A/QI triple $(G,X,H)$, the subgroup $H$ has \emph{finite height} in $G$ (see Proposition~\ref{prop:finiteheight}).  Finite height is a weakening of malnormality which has proved extremely useful in studying residual finiteness and cubulability of (relatively) hyperbolic groups.  See for example \cite{AGM09,Agol13,HruskaWise14}.

Restricting our attention to A/QI triples $(G,X,H)$ where $G$ is finitely generated, we obtain a stronger result.
\begin{restatable}{theorem}{stable}\label{thm:stable}
  Let $(G,X,H)$ be an A/QI triple, and suppose that $G$ is finitely generated. Then $H$ is stable in $G$.
\end{restatable}
\emph{Stability} is a strong quasi-convexity condition introduced by Durham--Taylor \cite{DurhamTaylor15} (see Definition \ref{def:stable}).  By work of Antol\'in--Mj--Sisto--Taylor \cite{AMST19}, stability implies finite height as well as the related properties of \emph{finite width} and \emph{bounded packing}.  Our arguments here are based on those of Sisto in \cite{Sisto16}, where he proves a stability result for hyperbolically embedded subgroups.  Note that in our setting, the subgroup $H$ may not  be hyperbolically embedded in $G$.  In particular, hyperbolically embedded subgroups have height at most one, whereas subgroups that are part of an A/QI triple often have height greater than one.

For our investigation of the boundary, we return to the setting of a general A/QI triple $(G,X,H)$.  For such a triple there is an associated action $G\acts \hat X$, where $\hat X$ is a hyperbolic space obtained from $X$ by equivariantly coning off the translates of an $H$--orbit.  (In case $X$ is the Cayley graph of $G$, this is the \emph{coned-off} or \emph{electrified} Cayley graph of $(G,H)$ studied, for example, in \cite{Bowditch12,Farb}.)  We characterize the boundary of the cone-off $\hat X$ as follows.
\begin{restatable}{theorem}{aqiboundary}\label{thm:aqiboundary}
    Suppose $(G,X,H)$ is an A/QI triple and fix $x\in X$.  Let $\mc{H}=\{gHx\mid gH\in G/H\}$.  Let $\hat{X}$ be the cone-off of $X$ with respect to $\mc{H}$.  Then $\partial \hat X$ is homeomorphic to the subspace of $\partial X$ obtained by deleting the limit sets of all the conjugates of $H$.
\end{restatable}
 
  In \cite{Ham16}, Hamenst\"adt proves an analogous result to Theorem~\ref{thm:aqiboundary} in the case that the hyperbolic space $X$ is strongly hyperbolic (as a metric space) relative to a collection of subspaces $\mathcal H$, without requiring the existence of a group action.
  
We use this description of $\partial \hat X$ to understand the geometry of the action $G\acts \hat X$.  In Corollary~\ref{cor:loxell}, we classify the elliptic, loxodromic, and parabolic elements for this action.  In  Proposition~\ref{prop:WPDpersists}, we describe the loxodromic WPD elements of $G\acts \hat X$; in particular we can show (Corollary~\ref{cor:wpd}) if $G\acts X$ is a WPD \emph{action}, then so is $G\acts \hat X$.

\subsection*{Outline} 
In Section~\ref{sec:prelim}, we recall some basic facts and fix some notation.  In Section~\ref{sec:newacyl} we give an alternative characterization of acylindricity along a subspace.  The following three sections can be read independently.
In Section~\ref{sec:finiteheight}, we prove some properties of A/QI triples.  In particular, we give short arguments establishing certain intersection properties of A/QI triples (see Propositions~\ref{prop:intersectedtriples} and~\ref{prop:finiteheight}).  We also prove Theorem~\ref{thm:combination}. 
In Section~\ref{sec:stability}, we recall the definition of stable subgroups and prove Theorem \ref{thm:stable}.
In Section~\ref{sec:coneoff}, we characterize the boundary of the coned-off space $\hat X$ by proving Theorem \ref{thm:aqiboundary}, along with the associated results about the action $G\acts \hat X$.    
Finally, in Section~\ref{sec:examples}, we use the combination theorem from Section~\ref{sec:finiteheight} to describe some further examples of A/QI triples.  Finally we pose several questions in Section~\ref{sec:questions}.

\subsection*{Acknowledgments}  We are extremely indebted to Sam Taylor for his comments on a previous version of this paper, and for telling us how to prove Proposition~\ref{prop:WPDpersists}.  We also thank Anthony Genevois for helpful comments about Artin and Cremona groups.  We also thank the referee for detailed comments which helped improve the exposition.

The second author thanks Daniel Groves, Mahan Mj, Bena Tshishiku, and Jacob Russell for useful conversations.  He was visiting Cambridge University for part of this work and thanks Trinity College and DPMMS for their hospitality.

The first author was partially supported by NSF Award DMS-1803368.  The second author was partially supported by Simons Collaboration Grant \#524176.

\section{Preliminaries}\label{sec:prelim}
In this section we recall some basic notions and fix some notation regarding Gromov products and the boundary of a hyperbolic space.  For more detail see, for example, \cite[III.H]{BH}.

\begin{definition}
Let $(X,d_X)$ be a hyperbolic metric space with basepoint $x_0$.  For any $x,y\in X$, the \emph{Gromov product of $x$ and $y$} is \[(x\mid y)_{x_0}=\frac12\left(d_X(x_0,x)+d_X(x_0,y)-d_X(x,y)\right).\]
\end{definition}
The following observation can be found, for example in Bridson--Haefliger~\cite[p. 410]{BH}.
\begin{lemma}
For any $\delta$--hyperbolic space $X$ and any $x,y\in X$, we have \[|(x\mid y)_{x_0}-d_X(x_0,[x,y])|\leq \delta,\] where $[x,y]$ is any geodesic from $x$ to $y$ and $x_0$ is a fixed basepoint.
\end{lemma}

A sequence of points $(x_k)$ in $X$ \emph{converges to infinity} if $\lim_{n,k\to\infty}(x_k\mid x_n)_{x_0}=\infty$.  An equivalence relation $\sim$ on sequences $(x_k)$, $(y_k)$ in $X$ which converge to infinity is defined as follows: $(x_k)\sim(y_k)$ if and only if $\lim_{k\to\infty}(x_k\mid y_k)_{x_0}=\infty$ if and only if $\lim_{n,k\to\infty}(x_k\mid y_n)_{x_0}=\infty$.

\begin{definition}
The \emph{Gromov boundary} $\partial X$ of a hyperbolic metric space $X$ is the set of equivalence classes of sequences in $X$ which converge to infinity.
\end{definition}

The Gromov product extends to $\bar X=X\cup \partial X$ by the formula
\begin{equation}\label{eq:gprodinfty}
  (x\mid y)_{x_0}=\sup\{\liminf_{i,j\to\infty} (x_i\mid y_j)_{x_0}\mid x_i\to x, y_i\to y\mbox{, and }x_i,y_i\in X\},
\end{equation}
for all $x,y\in \bar X$.  The union $\bar X$ is topologized by declaring that a sequence $(x_n)$ in $\bar X$ converges to a point $x\in\partial X$ if and only if \[\lim_{n\to\infty}(x_n\mid x)_{x_0}=\infty.\]

Although convergence to infinity, equivalence, and the Gromov boundary are defined using the basepoint $x_0$, they are actually independent of that basepoint.

The boundary can also be thought of in terms of quasi-geodesic rays, which are quasi-isometric embeddings of $[0,\infty)$.  
To simplify notation we will bundle together all the constants describing the quality of such a quasi-isometric embedding (and hence of such a quasi-geodesic) into a single number.

\begin{definition}\label{def:quasi-geodesic}
Let $\tau\geq 1$.  A map $f\from (X,d_X)\to (Y,d_Y)$ of metric spaces is a \emph{$\tau$--quasi-isometric embedding} if for all $x,y\in X$, we have $\tau^{-1}d_X(x,y)-\tau\leq d_Y(f(x),f(y))\leq \tau d_X(x,y)+\tau$.

  A \emph{$\tau$--quasi-geodesic} in $X$ is a $\tau$--quasi-isometric embedding of an interval into $X$, that is, a  map $\gamma\from I\to X$ where $I\subseteq \bR$ is connected and so that for all $s,t\in I$,

     $\tau^{-1}|s-t|-\tau\le d_X(\gamma(s),\gamma(t))\le \tau|s-t|+\tau$.

  \end{definition}
  
 If the map $\gamma\from I\to X$ is continuous, then we denote the length of the quasi-geodesic $\gamma$ in $X$ by $\ell_X(\gamma)$.   Given a discrete path (list of points) $\sigma=x_0,\dots, x_k$ in a metric space $X$, we define its length to be  $\ell^0(\sigma)=\sum_{i=0}^{k-1} d_X(x_i,x_{i+1})$.
 
 It is occasionally convenient to assume that a quasi-geodesic $\gamma$ is \emph{tame} in the sense of  \cite[III.H.1.11]{BH}.  When $\gamma$ is tame,  the map $\gamma\from I\to X$ is continuous, and we additionally have $\ell_X(\gamma|_{[s,t]})\leq \tau\, d_X(\gamma(s),\gamma(t)) + \tau$ for all $s,t\in I$.  Quasi-geodesics of uniform quality are uniformly closely fellow-traveled by tame quasi-geodesics (see \cite[III.H.1.11]{BH}).
 In this paper we will \emph{not} assume  quasi-geodesics are tame unless we explicitly state otherwise. 
   
  The following ``Morse'' property of quasi-geodesics in hyperbolic spaces is well-known (see e.g. \cite[III.H.1.7]{BH}).
  \begin{theorem}\label{thm:stability}
    For all $\delta\ge 0$, $\tau\ge 1$, there is an $M=M(\tau,\delta)\ge 0$ satisfying the following.  Let $X$ be a $\delta$--hyperbolic space, and let $\gamma_1$, $\gamma_2$ be $\tau$--quasi-geodesics with the same endpoints in $X\cup \partial X$.  Then the Hausdorff distance between $\gamma_1$ and $\gamma_2$ is at most $M$.
  \end{theorem}
  \begin{definition}
    Any number $M$ as in Theorem~\ref{thm:stability} will be called a \emph{Morse constant for the parameters $\tau,\delta$.}
  \end{definition}
  
  The following lemma follows  from the $2\delta$--slimness of geodesic quadrilaterals in a $\delta$--hyperbolic space.
  \begin{lemma}\label{lem:unifslimquads}
    Let $X$ be a $\delta$--hyperbolic space, let $a,b,a',b'$ be points of $X$, let $\gamma$ be a geodesic from $a$ to $b$, and let $\gamma'$ be a geodesic from $a'$ to $b'$.  Let $\varepsilon\geq \max\{d_X(a,a')),d_X(b,b')\}$ and suppose that $d_X(a,b) > 2\varepsilon + 2\delta$. Fix points $x,y$ on $\gamma$ that are distance exactly $\varepsilon$ from $a$ and $b$, respectively.  If $x',y'$ are the closest points on $\gamma'$ to $x$ and $y$, respectively, then the Hausdorff distance between the subsegment of $\gamma$ between $x$ and $y$, and the subsegment of $\gamma'$ between $x'$ and $y'$ is at most $6\delta$.
  \end{lemma}
  
  \begin{proof}
    We choose geodesics $[a,a']$ and $[b,b']$.
    Our choice of points $x,y$ on $\gamma$ ensures that $d_X(x,x')\leq 4\delta$ and $d_X(y,y')\leq 4\delta$.  To see this, notice that there is a point $z$ on $[a,a']\cup \gamma'\cup [b,b']$ at distance at most $2\delta$ from $x$.  If $z$ lies on $\gamma'$, then we are done.
    The lower bound on $d_X(a,b)$ implies that $z$ lies on $[a,a']$.  We have $d_X(z,a')=d_X(a,a')-d_X(z,a) \leq \varepsilon  - (\varepsilon-2\delta) =2\delta.$   In this case $d_X(a',x)\leq 4\delta$.  An analogous argument shows this holds for $y$. The result then follows from the $2\delta$--slimness of the geodesic quadrilateral with vertices $x,y,x',$ and $y'$.
  \end{proof}
  
  \begin{definition}
Let $\kappa\geq 0$.  A subset $Y$ of a metric space $X$ is \emph{$\kappa$--quasi-convex} if every geodesic in $X$ connecting points of $Y$ is contained in the $\kappa$--neighborhood of $Y$.
  \end{definition}

Given a quasi-geodesic ray $\gamma$ in $X$ starting at $x_0$, we obtain an equivalence class $[\gamma]\in\partial X$ by $[\gamma]:=[(a_i)]$ for any sequence of points $a_i$ on $\gamma$ with $\lim_{i\to\infty}d_X(x_0,a_i)=\infty$.  
\begin{remark}\label{rem:visibleboundary}By \cite[Remark~2.16]{KB}, there is a constant $\tau_0$ depending only on the hyperbolicity constant of $X$ such that given any point $\xi\in\partial X$, there is a $\tau_0$--quasi-geodesic ray $\gamma_\xi$ in X starting at $x_0$ such that $[\gamma_\xi]=\xi$.
\end{remark}

The Gromov product at infinity controls the Gromov product between points on quasi-geodesics, in the following sense.
\begin{lemma}\label{lem:gprodcontrol}
  For any $\tau\ge 1$, $\delta\ge 0$ there is some $E \ge 0$ so that the following holds.

  Let $X$ be $\delta$--hyperbolic, $x_0\in X$, and $\alpha,\beta\in \partial X$.  Let $\gamma_\alpha$ and $\gamma_\beta$ be $\tau$--quasi-geodesics starting at $x_0$ and tending to $\alpha$, $\beta$ respectively.  Then for all $s,t \ge 0$,
  \[ (\gamma_\alpha(s)\mid\gamma_\beta(t))_{x_0}\le (\alpha\mid\beta)_{x_0} + E .\]
\end{lemma}
\begin{proof}
  We fix some $s,t\ge 0$.
  From the definition of Gromov product at infinity~\eqref{eq:gprodinfty}, we see
  \begin{equation*}
    (\alpha\mid \beta)_{x_0} \ge \liminf_{i,j\to\infty}(\gamma_\alpha(i)\mid \gamma_\beta(j))_{x_0}.
  \end{equation*}
  Choosing $S\ge s$, $T\ge t$ sufficiently large, we can assume
  \begin{align}\label{eq:gprodineq}
    (\alpha\mid \beta)_{x_0} & \ge (\gamma_\alpha(S)\mid \gamma_\beta(T))_{x_0} - 1\nonumber \\
                             & = \frac{1}{2}\left(d_X(x_0,\gamma_\alpha(S)) + d_X(x_0,\gamma_\beta(T))-d_X(\gamma_\alpha(S),\gamma_\beta(T)\right) - 1
  \end{align}
  Let $M$ be a Morse constant for parameters $\tau$ and $\delta$.
  By the Morse property of quasi-geodesics we have $d_X(x_0,\gamma_\alpha(S))\ge d_X(x_0,\gamma_\alpha(s)) + d_X(\gamma_\alpha(s),\gamma_\alpha(S)) - 2M$  and  $d_X(x_0,\gamma_\beta(T))\ge d_X(x_0,\gamma_\beta(t)) + d_X(\gamma_\beta(t),\gamma_\beta(T)) - 2M$.  By the triangle inequality,
  \begin{equation*}
    d_X(\gamma_\alpha(S),\gamma_\beta(T)) \le d_X(\gamma_\alpha(S),\gamma_\alpha(s)) + d_X(\gamma_\beta(T),\gamma_\beta(t)) + d_X(\gamma_\alpha(s),\gamma_\beta(t)).
  \end{equation*}
  These three inequalities together with~\eqref{eq:gprodineq} yield
  \[(\alpha\mid \beta)_{x_0} \ge (\gamma_\alpha(s)\mid \gamma_\beta(t))_{x_0} - (2M + 1),\] establishing the lemma.
\end{proof}

Given a group acting by isometries on a $\delta$--hyperbolic space, we  define certain subsets of the boundary which arise as limits of sequences of elements from an orbit of a fixed base point.  
\begin{definition}\label{def:limitset}
  For $Y$ a subset of a $\delta$--hyperbolic space $X$, we write $\Lambda(Y)$ for the \emph{limit set} of $Y$.  This is
  \[
\Lambda(Y): = \{p\in \partial X \mid p=\lim_{n\to\infty} y_n \textrm{ for some sequence } y_n\in Y\}.
\]
If we are given a group $G$ acting by isometries on $X$, we define the limit set $\Lambda(T)$ of a subset $T$ of $G$ to be $\Lambda(\{tx_0\mid t\in T\})$ where $x_0$ is some base point.  (The limit set does not depend on this base point.)
\end{definition}

\section{An alternative characterization of acylindricity} \label{sec:newacyl}

When a subgroup $H<G$ is quasi-isometrically embedded by an action $G\curvearrowright X$ on a hyperbolic space, there is an alternate characterization of acylindricity of $G\curvearrowright X$ along $H$.  In practice, this alternate characterization is easier to check.

\begin{definition} \label{def:newacyl}
  Let $G$ act by isometries on a hyperbolic metric space $X$, and let $H$ be a subgroup  that  is  quasi-isometrically embedded by the action.  The action of $G$ on $X$ is \emph{acylindrical along $H$} if for some (equivalently any) $H$--cobounded $Y\subset X$ the following condition holds:
  For any $\varepsilon\geq 0$, there exist constants $D=D(\varepsilon, Y)$ and $N=N(\varepsilon,Y)$ such that any collection of distinct cosets $\{g_iH,\dots g_kH\}$ such that
\[
\diam\left(\bigcap_{i=1}^k \mathcal N_\varepsilon(g_iY)\right)>D
\] 
has cardinality at most $N$.
\end{definition}

For a fixed $\varepsilon>0$ and $H$--cobounded $Y\subset X$, we call the constants $D(\varepsilon,Y),N(\varepsilon,Y)$ the \emph{constants of acylindricity} of the action of $G$ on $X$ along $H$.

The following theorem shows that this definition is equivalent to Definition \ref{def:acyl}.

\begin{theorem}\label{thm:newacyl}
  Let $G$ act by isometries on a hyperbolic space $X$, and let $H$ be a subgroup  that is quasi-isometrically embedded by the action.  Let $Y$ be an $H$--cobounded subset of $X$.   Then the action of $G$ is acylindrical along $H$ in the sense of Definition~\ref{def:acyl} if and only if it is acylindrical along $H$ in the sense of Definition~\ref{def:newacyl}.
\end{theorem}

\begin{proof}
Let $x\in X$, and fix $Y = Hx$.

We prove the backward direction first and assume that the action of $G$ on $X$ is acylindrical along $H$ in the sense of Definition \ref{def:newacyl}. Fix $\varepsilon\geq 0$, and let $D=D(2\varepsilon,Y)$ and $N=N(2\varepsilon,Y)$.  Since $H$ is quasi-isometrically embedded by the action, there exists a constant $P>0$ such that there are at most $P$ elements $h\in H$ such that $d_X(x,hx)\leq 2\varepsilon$.

Fix $h_1,h_2\in H$ satisfying $d_X(h_1x,h_2x)>D$, and let 
\[
\Delta=\{g\in G\mid d_X(h_1x,gh_1x)\leq \varepsilon \textrm{ and } d_X(h_2x,gh_2x)\leq \varepsilon\}.
\]
We claim that $\#\Delta\leq NP$.  Towards a contradiction, assume this is not the case, so that there exist distinct  $\{g_1,\dots,g_k\}\subset \Delta$ with $k>NP$.  Then
\[
\diam\left(\bigcap_{i=1}^k \mathcal N_\varepsilon(g_iHx)\right)>D.
\]
By assumption, the number of \emph{distinct cosets} in $\{g_1H,\ldots,g_kH\}$ is at most $N$.  Rearranging the list, we can therefore assume that the first $(P+1)$ cosets are the same.  In other words $g_1^{-1}g_j\in H$ for all $j\le P+1$.
Since the $g_j$ are all in $\Delta$ we have $d_X(h_1x,g_jh_1x)\leq \varepsilon$ for each $j$.  Let $\hat{g}_j = h_1^{-1}g_1^{-1}g_jh_1$, and note that $\hat{g}_j\in H$ for $j\le P+1$.  Moreover
\begin{equation}\label{eq:hatbound}
d_X(x,\hat{g}_j x)=d_X(g_1h_1x,g_jh_1x)\leq 2\varepsilon.
\end{equation}
There are only $P$ elements of $H$ satisfying~\eqref{eq:hatbound}, so we have  $\hat{g}_j = \hat{g}_i$ for some $i<j\le P+1$.  But this implies $g_j = g_i$, which is a contradiction.
We have therefore shown the action is acylindrical along $Y$ in the sense of Definition~\ref{def:acyl}, with $R = D$ and $M = N P$.

We now prove the forward direction and assume the action is acylindrical along $H$ in the sense of Definition \ref{def:acyl}.  Let $\varepsilon>0$, let $\delta$ be a hyperbolicity constant for $X$, and suppose $H$ is $\tau$--quasi-isometrically embedded by the orbit map.   The Morse property of quasi-geodesics in hyperbolic spaces implies that there is a constant $\lambda$ depending only on $\tau$ and $\delta'$ such that the orbit $Hx$ is $\lambda$--quasi-convex in $X$.  
Let $R = R(\varepsilon')$, $M = M(\varepsilon')$, be the constants provided by Definition \ref{def:acyl}, where
\begin{equation*}
  \varepsilon' = 4(\varepsilon + \delta + \lambda).
\end{equation*}
We will show that Definition \ref{def:newacyl} holds with constants 
\begin{equation*}
  D = D(\varepsilon,Y) = R + \varepsilon'+ \varepsilon,\quad N = N(\varepsilon,Y) =  MB + 1,
\end{equation*}
where $B$ is the number of elements $h$ of $H$ with $d_X(hx,x)\le R + \varepsilon'$.

Suppose that $g_0 H,\ldots ,g_n H$ are distinct cosets with
\begin{equation*}
  \diam\left(\bigcap_{i=0}^n \mathcal N_\varepsilon(g_i H x)\right) > D.
\end{equation*}
We may translate this intersection so that it contains points close to $x$ and to some $kx\in Hx$ which is distance at least $D-\varepsilon$ from the identity.  More specifically, we can assume the following, possibly after changing the coset representatives $g_i$:
\begin{enumerate}
\item $g_0 = 1$;
\item $d_X(g_ix,x)\le 2\varepsilon$ for each $i$; and
\item there is some $k\in H$ so that $d_X(x,kx)> D - \varepsilon$ and so that for each $i$ there is some $k_i \in H$ so that $d_X(g_ik_i x,k x)\le 2\varepsilon$.
\end{enumerate}

Let $a$ lie on a geodesic from $x$ to $k x$ so that $d_X(x,a) = R + 2(\delta + \lambda + \varepsilon)$.  By quasi-convexity of $Hx$, there is some $h_0\in H$ with $d_X(h_0 x,a)\le \lambda$.

Let $i\le n$.  Since $d_X(a,\{x,kx\})\geq \varepsilon'/2>2\varepsilon+2\delta$,  the $2\delta$--slimness of quadrilaterals implies that there is some $a_i$ on a geodesic from $g_i x$ to $g_i k_i x$ so that $d_X(a,a_i)\le 2\delta$.  Again using quasi-convexity, there is some $h_i\in H$ with $d_X(g_ih_ix,a_i)\le \lambda$.  For each $i$, we have
\begin{equation*}
  d_X(h_0x,g_ih_ix)\le 2\lambda + 2\delta.
\end{equation*}
Using the triangle inequality, for each $i$ we also have
\begin{align*}
d_X(x,h_i x) & = d_X(g_ix, g_i h_i x)  \\
&\geq d_X(x, a) -(d_X(x, g_i x)+d_X(a, a_i) + d_X(a_i, g_i h_i x) ) \\
& \geq R+2(\delta+\lambda+\varepsilon) -(2\varepsilon + 2\delta + \lambda) \\
& = R + \lambda,
\end{align*}
 and 
 \begin{align*}
 d_X(g_ix, g_i h_ix)& \leq d_X(x, a) + d_X(x, g_i x)+d_X(a, a_i) + d_X(a_i, g_i h_i x) \\
 &\leq R+2(\delta+\lambda+\varepsilon) +2\varepsilon + 2\delta + \lambda \\
 & =R + 4\varepsilon + 4\delta + 3\lambda.
 \end{align*}
 Putting these together, we obtain
\begin{equation}\label{eq:distinterval}
  d_X(x,h_i x) \in [ R + \lambda, R + 4\varepsilon + 4\delta + 3\lambda ].
\end{equation}
There are at most $B$ elements of $H$ satisfying \eqref{eq:distinterval}, so it remains to fix some such element $h_i$ and bound the number of $g_j$ associated to it.  But for any two such $g_j,g_k$, we have $d_X(g_jx,g_k x)\le \varepsilon'$ and $d_X(g_jh_ix,g_kh_ix)\le \varepsilon'$.  Acylindricity tells us there are at most $M = M(\varepsilon')$ such elements, so over all possible $h$, we have at most $MB$ such.  Thus the action is acylindrical along $H$ in the sense of Definition~\ref{def:newacyl}.  
\end{proof}

\begin{definition}\label{def:many subgroups}
  Let $\mathcal H$ be a collection of subgroups of $G$, and suppose $G\curvearrowright X$ is an action on a hyperbolic space so that each $H\in\mathcal H$ is quasi-isometrically embedded by the action.  Suppose there is a uniformly quasi-convex collection of subspaces $\{Y_H\mid H\in \mc{H}\}$ so that $H$ acts coboundedly on $Y_H$ for each $H\in \mc{H}$.  Then the action is \emph{acylindrical along $\mathcal H$} if for all $\varepsilon\geq 0$ there exist constants $D=D(\varepsilon, \{Y_H\}_{H\in\mathcal H})$ and $N=N(\varepsilon,\{Y_H\}_{H\in\mathcal H})$ such that any collection $\mathcal S\subseteq \{gY_H \mid H\in\mathcal H, g\in G\}$ satisfying
\[
\diam\left(\bigcap_{C\in \mathcal S} \mathcal N_\varepsilon(C)\right)>D
\] 
  has cardinality at most $N$.
\end{definition}

Using this characterization of acylindricity along a collection of quasi-isometrically embedded subgroups, we can generalize the notion of an A/QI triple to allow a collection of subgroups.

\begin{definition}
An \emph{A/QI triple} $(G,X,\mathcal H)$ consists of a Gromov hyperbolic space $X$ which is acylindrical along the collection $\mathcal H$ of infinite subgroups of $G$ in the sense of Definition~\ref{def:many subgroups}.
\end{definition}

The following proposition shows that the A/QI condition for \emph{finite} collections of subgroups reduces to the A/QI condition for a single subgroup.
\begin{proposition}\label{prop:finitecollection}
Let $\mathcal H=\{H_1,\dots,H_n\}$ be a finite collection of infinite subgroups of $G$. Then $(G,X,\mathcal H)$ is an A/QI triple if and only if $(G,X,H_i)$ is an A/QI triple for all $1\leq i\leq n$.  
\end{proposition}

\begin{proof}
This follows immediately from the definition of an A/QI triple and the fact that $\mc H$ is a finite collection of subgroups.
\end{proof}

\section{Intersections, combinations, and finite height}\label{sec:finiteheight}
In this section, we establish basic properties of A/QI triples.   We \emph{do not} assume in this section that $G$ is finitely generated.

\subsection{Intersections}
The first result of this section implies, in particular, that if $(G, X, A)$ and $(G, X, B)$ are A/QI triples for the same action $G\acts X$, then $(G, X, A \cap B)$ is an A/QI triple.

\begin{proposition}\label{prop:intersectedtriples}
Suppose $(G, X, A)$ is an A/QI triple and $B<G$ is an infinite subgroup quasi-isometrically embedded by the action.  Then $(G,X,A\cap B)$ is an A/QI triple.
\end{proposition}
\begin{proof}
  Let $\delta$ be a constant of hyperbolicity for $X$, and let $x\in X$.  Since $A$ and $B$ are quasi-isometrically embedded by the action, there is some $\lambda>0$ so that the sets $Ax$ and $Bx$ are $\lambda$--quasi-convex.  
  There is also a constant $M$ so that any geodesic with endpoints in $Bx$ lies Hausdorff distance at most $M$ from the image of a geodesic in some fixed Cayley graph for $B$.
  
  Let $C = A\cap B$.  Since $Ax$ and $Bx$ are proper subspaces of $X$, $Cx$ is also a proper subspace of $X$.  Thus to prove $C$ is quasi-isometrically embedded by the orbit, it suffices to show that $Cx$ is quasi-convex in $X$.  We suppose by contradiction that it is not.  There is therefore a sequence of elements $\{c_i\}$ of $C$ so that geodesics $\gamma_i$ from $x$ to $c_ix$ travel arbitrarily far away from $Cx$.  In other words, there are points $y_i\in \gamma_i$ with $\lim_{i\to\infty}d_X(y_i,Cx)=\infty$.  Since $Ax$ and $Bx$ are $\lambda$--quasi-convex, there are, for each $i$, $a_i\in A$ and $b_i\in B$ so that both $d_X(a_ix,y_i)$ and $d_X(b_ix,y_i)$ are at most $\lambda$.  Translating by $b_i^{-1}$ we obtain a sequence of geodesics $b_i^{-1}\gamma_i$ joining $b_i^{-1}x$ to $b_i^{-1}c_i$.  The geodesic $b_i^{-1}\gamma_i$
  is Hausdorff distance at most $M$ from some quasi-geodesic $\sigma_i$
    which is the image of a geodesic $\tilde{\sigma}_i$ joining $b_i^{-1}$ to $b_i^{-1}c_i$ in a Cayley graph for $B$.  Passing to a subsequence we may assume the geodesics $\tilde{\sigma}_i$ converge to a bi-infinite geodesic $\tilde{\sigma}$ in the Cayley graph for $B$, whose image in $X$ is a bi-infinite quasi-geodesic $\sigma$.  

  Let $\varepsilon = M+\delta + 2\lambda$, and let $D =\max\{\varepsilon+1, D(\varepsilon,Ax)\}$, where $D(\varepsilon,Ax)$ is the constant from Theorem~\ref{thm:newacyl}, applied with $H = A$.  Fix a point $p$ on the quasi-geodesic $\sigma$ so that $d_X(x,p) \ge D$ and so that $p$ is within $M$ of the subsegment of $b_i^{-1}\gamma_i$ joining $b_i^{-1}y_i$ to $b_i^{-1}c_ix$ for all but finitely many $i$.  See Figure~\ref{fig:intersect}.
  \begin{figure}[htbp]
    \centering
    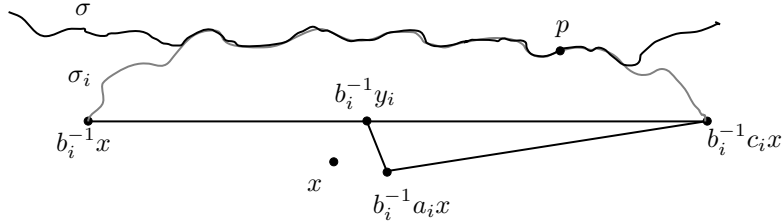
    \caption{The point $p$ is on the bi-infinite quasi-geodesic $\sigma$, at distance $D$ from $x$.  The index $i$ is assumed to be large enough that $\sigma_i$ contains $p$.}
    \label{fig:intersect}
  \end{figure}

  Discarding finitely many $i$, we may assume that $p$ is within $M$ of all these subsegments.  Let $p_i'$ be a nearest point to $p$ on the subsegment of $b_i^{-1}\gamma_i$ joining $b_i^{-1}y_i$ to $b_i^{-1}c_ix$, and consider a geodesic triangle with one side equal to this subsegment, and with the third corner equal to $b_i^{-1}a_i x$.  Let $p_i$ be a point within $\delta$ of $p_i'$ on one of the other two sides of this triangle.  We have $d_X(p,p_i)\le M+\delta$.  If $p_i$ lay on the side joining $b_i^{-1}a_ix$ to $b_i^{-1}y_i$, this would imply $d_X(p,x) \le M + 2\lambda + \delta$, a contradiction, since $d_X(p,x) \ge \varepsilon + 1$.  Thus $p_i$ lies on a geodesic joining  $b_i^{-1}a_ix$ to $b_i^{-1}c_ix$. 
Quasi-convexity of $Ax$ gives us an $\hat{a}_i\in A$ so that $b_i^{-1}\hat{a}_ix$ is at most $\lambda$ from $p_i$.  Thus we have $d_X(b_i^{-1}\hat{a}_ix,p)\le M+\delta+\lambda <\varepsilon$ and $d_X(b_i^{-1}a_ix,x)\le 2\lambda <\varepsilon$.  In particular, $x$ and $p$ are in $N_\varepsilon (b_i^{-1}Ax)$ for all $i$, so we have
\begin{equation*}
  \diam \left( \bigcap_i \mathcal N_\varepsilon (b_i^{-1}Ax)\right)\ge D.
\end{equation*}
By Theorem~\ref{thm:newacyl}, infinitely many of the cosets $b_i^{-1}A$ coincide.  Discarding terms and relabeling we may assume that $b_i^{-1}A = b_0^{-1}A$ for all $i$.  In particular we have $b_ib_0^{-1}\in C$ for every $i$.
But
\[ d_X(b_ib_0^{-1}x,y_i)\le d_X(b_i b_0^{-1} x, b_i x)+ d_X(b_i x, y_i) \le  d_X(x,b_0 x) + \lambda,\]
so $d_X(b_ib_0^{-1}x,y_i)$ is bounded independently of $i$.
This contradicts our choice of $y_i$, showing that in fact $Cx$ is quasi-convex.

The orbit $Cx$ is a subset of $Ax$, along which $G$ acts acylindrically, so $G$ acts acylindrically along $Cx$, and $(G,X,C)$ is an A/QI triple as desired.
\end{proof}

\subsection{A combination theorem}
In this subsection we prove Theorem \ref{thm:combination}, which we restate for the convenience of the reader.

\combination*

We note that by Proposition \ref{prop:finitecollection}, the assumption that each $(G,X,H_i)$ is an A/QI triple is equivalent to assuming that $(G,X,\{H_1,\dots, H_k\})$ is an A/QI triple.

Before proving the theorem, we need some preliminary lemmas.
We say that a path $q$ in a metric space $X$ is a \emph{piecewise quasi-geodesic} if $q$ can be expressed as a concatenation $q_1\cdot \ldots \cdot q_n$ where each $q_i$ is a continuous quasi-geodesic path.  The following lemma gives conditions under which a piecewise quasi-geodesic is a quasi-geodesic (with constants independent of the number of quasi-geodesic segments being concatenated).

\begin{lemma}[{\cite[Lemma 4.2]{Minasyan}}] \label{lem:brokenqgeod}
Let $x_0, x_1,\dots, x_n$ be points in a $\delta$--hyperbolic space $X$ and let $q_i$ be a continuous $\tau$--quasi-geodesic  from $x_{i-1}$ to $x_i$. Then for any $C_0\geq 14\delta$ and for $C_1=12(C_0+\delta)+\tau+1$, if $\ell_X(q_i)\geq \tau C_1$ and $(x_{i-1}\mid x_{i+1})_{x_i}\leq C_0$ for all $i$, then the concatenation $q_1\cdot \ldots\cdot q_n$ is a $\max\{4\tau, \frac52M+C_1\}$--quasi-geodesic, where $M$ is the Morse constant for the parameters $\tau,\delta$.
\end{lemma}

Let $G,X,$ and $H_1,\dots, H_k$ be as in the statement of Theorem \ref{thm:combination}.  We fix the following data.

 Let $T_i$ be a finite generating set for $H_i$.  Let $\delta>1$ be a hyperbolicity constant which suffices both for $X$ and for the Cayley graphs $\Cay(H_i,T_i)$.  We fix a basepoint $x\in X$ and let $\pi\from G\to X$ be the orbit map defined by $g\mapsto gx$. Since each $(G,X,H_i)$ is an A/QI triple, there is a constant $\tau$ such that each $H_ix$ is $\tau$--quasi-isometrically embedded in $X$.  Let $E$ be the constant from Lemma \ref{lem:gprodcontrol} with this choice of $\tau$ and $\delta$.  

Since $\Lambda(H_i)\cap \Lambda(H_j)=\emptyset$ for all $i\neq j$, there is a constant $\kappa$ such that 
\begin{equation}\label{eqn:Hkappa}
(\xi \mid \xi')_{x}<\kappa
\end{equation}
for all $\xi\in \Lambda(H_i)$ and $\xi'\in \Lambda(H_j)$.

 Let 
 \[
 \eta=\max\{d_X(x,tx)\mid i=1,\dots, k \textrm{ and } t\in T_i\},
 \]
  which is well-defined since  $\#T_i<\infty$ for each $i$.

For each $i$ and each $t\in T_i$, fix a geodesic $[x,tx]$ in $X$.  Since each $H_ix$ is $\tau$--quasi-isometrically embedded in $X$, for any  $i$ and any geodesic word $t_1\dots t_\ell$ where $t_j\in T_i$, the concatenation 
\[
[x,t_1x]\cdot [t_1x,t_1t_2x]\cdot \ldots \cdot [t_1\dots t_{\ell-1}x,t_1\dots t_\ell x]
\]
is a $\tau$--quasi-geodesic in $X$.  Let $M$ be the Morse constant for parameters $\tau$ and $\delta$. \\

\putinbox{For the remainder of the subsection, the following data is fixed: the sets $T_i$, the basepoint $x$, the map $\pi$, the constants $\delta$, $\tau$, $E$, $\kappa$, $\eta$, and $M$, and the geodesics $[x,tx]$.}
\\

Given any collection of subgroups $\mc{K} = \{K_i\leq H_i\}$, define a metric graph $\Gamma(\mc{K})$ as follows.  

\begin{definition}\label{def:Gamma}
Let $\Cay(H_i,T_i)$ be the Cayley graph of $H_i$ with respect to the finite generating set $T_i$.  We will build $\Gamma(\mc{K})$ inductively.  At the base stage, we begin with a single Cayley graph $\Cay(H_1,T_1)$, and identify $K_1$ with the orbit of the identity under $K_1$.  At each point in $K_1\subset \Cay(H_1,T_1)$, attach a disjoint copy of $\Cay(H_j,T_j)$ for each $j\neq 1$ along the identity $1\in\Cay(H_j,T_j)$. That is, for each $h\in K_1$ and each $j\neq 1$, we identify the identity  $1$ in a disjoint copy of  $\Cay(H_j,T_j)$ with the vertex $h\in \Cay(H_1,T_1)$. At each subsequent step, we  repeat the same process for each new Cayley graph $\Cay(H_j,T_j)$, by attaching a disjoint copy of $\Cay(H_i,T_i)$ for each $i\neq j$ to each point in $K_j\subset\Cay(H_j,T_j)$ along the identity $1\in\Cay(H_i,T_i)$.  

The resulting graph $\Gamma(\mc{K})$ is the smallest subgraph of $\Cay(H_1*\cdots*H_k,\cup_{i=1}^k T_i)$ which is invariant under the action of $K_1*\cdots * K_k \le H_1*\cdots *H_k$ and contains $\Cay(H_i,T_i)$ for each $i$.

Vertices in $\Gamma(\mc{K})$ come in two types: those that lie in the orbit of the identity under $K_i$ in some copy of $\Cay(H_i,T_i)$, and those that do not.  We call vertices of the first type \emph{wedge vertices}.  

Each edge in $\Gamma(\mc{K})$ is an edge of some $\Cay(H_i,T_i)$ and so is labeled by an element $t\in T_i$.  We define the length of such an edge to be $d_X(x,tx)$.
\end{definition}

Define a map $\psi_{\mc{K}}\from \Gamma(\mc{K})\to X$ by sending each edge whose initial vertex is the identity and whose label is  some $t\in T_i$ to the geodesic $[x,tx]$ chosen above, and extending equivariantly by the action of $K_1*\cdots *K_k$ on $\Gamma(\mc{K})$.  By definition, $\psi_{\mc{K}}$ restricts to an isometry on each edge.  The metric graph $\Gamma(\mc{K})$ and the map $\psi_{\mc{K}}$ depend on the choice of  subgroups $\mc{K}$.  However, we will show that for certain choices of subgroups, the map $\psi_{\mc{K}}$ is a uniform quasi-isometric embedding.
\\

\putinbox{We also fix constants
  \begin{align*}
    C_0 & = \kappa+E+14\delta + \frac{3}{2}(2 \tau\delta + \tau) \\
    C_1 & =12(C_0 + \delta)+\tau+1 \\
    \mu & = \max\{4\tau, \frac52M+C_1\} + \tau C_1       
  \end{align*}
and let
$\mathcal S_1=\{h\in \bigcup_{i=1}^k H_i\setminus\{1\}\mid d_X(x,hx)\leq \tau C_1\}$.
}\\

\begin{lemma}\label{claim:qiemb}
Let $\mc{K} = \{K_i\leq H_i\}$ be a family of quasi-isometrically embedded subgroups satisfying 
\[
K_i\cap \mathcal S_1=\emptyset,
\] and let $\Gamma(\mc{K})$ be the graph from Definition \ref{def:Gamma} with this choice of subgroups $K_i$.
The map $\psi_{\mc{K}}\from \Gamma(\mc{K})\to X$ is a $\mu$--quasi-isometric embedding which takes every geodesic in $\Gamma(\mc{K})$ to a tame $\mu$--quasi-geodesic.
\end{lemma}

\begin{proof}
  For ease of notation, we write $\Gamma$ for $\Gamma(\mc{K})$ and $\psi$ for $\psi_{\mc{K}}$.
A geodesic edge path in $\Gamma$ is either completely contained in a copy of $\Cay(H_i,T_i)$ for some $i$ or passes through multiple Cayley graphs.  In the first case, the image of the path will be a $\mu$--quasi-geodesic because each $\Cay(H_i,T_i)$ is $\tau$--quasi-isometrically embedded in $X$ and $\mu>\tau$.  

Thus we focus on geodesic edge paths in $\Gamma$ which pass through multiple Cayley graphs.
Note that such a path $\sigma$ can only change Cayley graphs at wedge vertices, which are necessarily in the orbit of the identity under some $K_i$.  Thus $\sigma$ is the concatenation of geodesic edge paths 
\begin{equation}\label{eqn:geodpath}
\sigma_1\cdot \sigma_2\cdot \ldots\cdot \sigma_p,
\end{equation}
where  each $\sigma_\ell$ is a maximal geodesic edge path contained in some $\Cay(H_{i_\ell},T_{i_\ell})$ with $H_{i_\ell}\neq H_{i_{\ell+1}}$.  Note that the terminal vertex of $\sigma_1$ is a wedge vertex,  the initial vertex of $\sigma_p$ is a wedge vertex, and, if $2\leq \ell\leq p-1$, then both endpoints of $\sigma_\ell$ are wedge vertices.  Moreover,  for $2\leq \ell\leq p-1$, the label of each $\sigma_\ell$ is an element of $K_{i_\ell}$ by construction.

Our main tool for understanding the image of such paths in $X$ will be Lemma \ref{lem:brokenqgeod}, applied with the constants $C_0,C_1$ defined just before the statement of the lemma.

For each $i$ and each  $h\in K_i\setminus\{1\}$,  we have  
 \begin{equation}\label{eqn:Hedgelength}
d_X(x,hx)>\tau C_1,
\end{equation}
by our choice of subgroups $K_i$.

Fix some $\ell$.  By translating by an element of $G$ we may assume the image of the terminal point of $\sigma_\ell$ (and therefore the image of the initial point of $\sigma_{\ell+1}$) is $x$.  Let $z_1$ be the initial point of $\sigma_\ell$ and $z_2$ the terminal point of $\sigma_{\ell+1}$.  Let $y_1 = \psi(z_1)$ and $y_2 = \psi(z_2)$.
We have $z_1\in H_\ell$, and $z_2\in H_{\ell+1}$.  Since $H_\ell$ is an infinite $\delta$--hyperbolic group, there is some geodesic ray $r_1$ in $H_\ell$ starting at the identity and passing within $2 \delta$ of $z_1$ (see for example \cite[Lemma 3.1]{BestvinaMess}).  Similarly there is a geodesic ray $r_2$ in $H_{\ell+1}$ starting at the identity and passing within $2 \delta$ of $z_2$.  For $i\in \{1,2\}$ the image $\psi\circ r_i$ is a $\tau$--quasi-geodesic in $X$, starting at $x$ and passing through some point $y_i'$ within $2\tau\delta+\tau$ of $y_i$. These quasi-geodesic rays limit to points $\xi_1 \in \Lambda(H_\ell)$ and $\xi_2\in\Lambda(H_{\ell + 1})$, respectively.
We find (using Lemma~\ref{lem:gprodcontrol} in the second line)
\begin{align*}
  (y_1\mid y_2)_{x} & \le (y_1'\mid y_2')_{x} + \frac{3}{2}(2 \tau\delta + \tau)\\
                    & \le (\xi_1\mid \xi_2)_{x}+E + \frac{3}{2}(2 \tau\delta + \tau) \\
  & \le \kappa +E + \frac{3}{2}(2 \tau\delta + \tau)< C_0.
\end{align*}
Combining this bound with the inequality~\ref{eqn:Hedgelength} and Lemma~\ref{lem:brokenqgeod} we see that if $\sigma_1$ and $\sigma_p$ are empty, then $\sigma$ is mapped by $\psi$ to a $\max\{4\tau, \frac52M+C_1\}$--quasi-geodesic (and hence a $\mu$--quasi-geodesic, since $\mu = \max\{4\tau, \frac52M+C_1\}+\tau C_1$).

Suppose $\sigma_p=\emptyset$, while $\sigma_1\neq\emptyset$, and suppose the initial vertex of $\sigma_1$ is an arbitrary vertex $a$ of $\Cay(H_{i_1},T_{i_1})$ and  the terminal vertex is a wedge vertex $b$.  If $d_X(ax,bx)>\tau C_1$ then the argument above shows that $\sigma$ is mapped by $\psi$ to a $(\mu-\tau C_1)$--quasi-geodesic.  On the other hand, if $d_X(ax,bx)\leq \tau C_1$, then $\sigma$ is mapped by $\psi$ to a $\mu$--quasi-geodesic in $X$.  A similar argument holds when $\sigma_2\neq\emptyset$ and when both $\sigma_1$ and $\sigma_2$ are non-empty.  In any case, $\sigma$ is mapped by $\psi$ to a $\mu$--quasi-geodesic in $X$.

Since $\psi$ restricted to any geodesic in $\Gamma$ gives a (tame) $\mu$--quasi-geodesic, the entire map $\psi$ is a $\mu$--quasi-isometric embedding. 
\end{proof}
\begin{corollary}\label{cor:QIFreeProd}
  Under the assumptions of Lemma~\ref{claim:qiemb} the subgroup $K =\langle\bigcup\mc{K}\rangle$ is isomorphic to the free product $K_1\ast\cdots\ast K_k$ and is quasi-isometrically embedded by the action $G\acts X$.
\end{corollary}
\begin{proof}
  Let $\phi\from K_1\ast\cdots\ast K_k\to K$ be the canonical map.  It is clearly surjective.   To see that it is injective, suppose $f$ is a non-trivial element of $K_1\ast \cdots \ast K_k$.   If $f\in K_i$ for some $i$ then $\phi(f)\neq 1$, by our choice of $K_i$, so $f$ must have a normal form involving at least two factors.  It is then straightforward to see that $f$ has a bi-infinite geodesic axis in $\Gamma$.  By Lemma~\ref{claim:qiemb}, this bi-infinite geodesic has image a bi-infinite $\mu$--quasi-geodesic in $X$, on which $\langle \phi(f)\rangle$ acts co-compactly.  In particular, $\phi(f)$ is non-trivial.

  The free product $K_1\ast \cdots \ast K_k$ is quasi-isometrically embedded by its action on $\Gamma_{\mc{K}}$, since each $K_i$ is undistorted in $H_i$.  Again using Lemma~\ref{claim:qiemb} this implies that the orbit map $k \mapsto k x$ quasi-isometrically embeds $K$ in $X$.
\end{proof}

We are now ready to prove Theorem \ref{thm:combination}.

\begin{proof}[Proof of Theorem \ref{thm:combination}]
We continue to use the data fixed after Lemma \ref{lem:brokenqgeod}.
 Let $M'$ be the Morse constant for parameters $\mu$ and $\delta$, and 
fix
\[
\varepsilon'=2M' + M + 10\delta.
\] 
The action of $G$ of $X$ is assumed to be acylindrical along $\{H_1,\dots, H_k\}$.  For each $H_i$, let $Y_i$ be $\psi(\Cay(H_i,T_i))$, and 
let $D=D(\varepsilon',\{Y_1,\dots, Y_k\})$ and  $N=N(\varepsilon', \{Y_1,\dots, Y_k\})$ be the constants of acylindricity of this action, as described in Definition~\ref{def:many subgroups}.

Fix a constant 
\begin{equation}\label{eqn:C}
C> \max\{ \mu(2D+8\delta+2M'+1), \tau C_1\},
\end{equation}
and let 
\[
\mathcal S=\{h\in \bigcup_{i=1}^k H_i\setminus\{1\}\mid d_X(x,hx)\leq C\}.
\]

Since the subgroups $H_i$ are uniformly quasi-isometrically embedded in $X$ by the orbit map, there is a constant $P$ so that for any $i$ and any $h\in H_i$, the number of elements $h'\in H_i$ so that $d_X(hx,h'x)\leq \frac{C}{\mu} +2M+2M'+12\delta-1$ is at most $P$.

Let $\mc{K} = \{K_i\leq H_i\}$ be a family of quasi-isometrically embedded subgroups satisfying
\[
K_i\cap \mathcal S=\emptyset,
\] and let $\Gamma = \Gamma_{\mc{K}}$ be the graph from Definition \ref{def:Gamma} with this choice of subgroups $K_i$.  Since  $C > \tau C_1$, we have $\mathcal S_1 \subset \mathcal S$.   Corollary~\ref{cor:QIFreeProd} therefore implies that $K\cong K_1\ast \cdots \ast K_k$ and that $K$ is quasi-isometrically embedded by the action of $G$ on $X$.

It remains to show that $G$ acts acylindrically along $K$.  
Fix $\varepsilon>0$, and let 
\[
R=\frac{C}{\mu}-1+2\varepsilon.
\]
Let $\Delta=\{gK\}$ be a collection of distinct cosets satisfying 
\[
\diam\left(\bigcap_{gK\in \Delta}\mathcal N_\varepsilon(gKx)\right)>R.
\]
We will give a uniform bound on $\#\Delta$; namely we will show $\#\Delta\le 2NP$.  To do so, we will show that if $\#\Delta > 2NP$, we contradict the acylindricity of the action of $G$ on $X$ along $\{H_1,\dots, H_k\}$.

Let $\gamma$ be a geodesic segment in $X$ of length at least $R$ whose endpoints are contained in the $\varepsilon$--neighborhood of $gKx$ for all $gK\in \Delta$.
We  consider an arbitrary $gK\in \Delta$ but for each such coset fix a representative $g$.
By Lemma~\ref{claim:qiemb}, there is a tame $\mu$--quasi-geodesic  $\beta_g$  in $g\psi(\Gamma)$ with endpoints within $\varepsilon$ of the endpoints of $\gamma$. 

Let $\gamma'$  be an arbitrary subgeodesic of $\gamma$  with endpoints at distance at least $\varepsilon$ from the endpoints of $\gamma$ and whose length is exactly $\frac{C}{\mu}-1-2M'-8\delta$.  By our choice of $R$, such a subgeodesic exists. 
By Lemma~\ref{lem:unifslimquads}, there is a subsegment $\xi_g$ of the geodesic joining the endpoints of $\beta_g$ which has endpoints within $4\delta$ of those of $\gamma'$ and which is Hausdorff distance at most $6\delta$ from $\gamma'$. The segment $\xi_g$ is Hausdorff distance at most $2M'+2\delta$  from a subpath $\beta_g'$ of $\beta_g$ whose endpoints are within $M'$  of those of $\xi_g$.  To see this, note that we can find points on $\beta_g$ at distance at most $M'$ from the endpoints of $\xi_g$ by the Morse property of quasi-geodesics.  Now consider the geodesic quadrilateral formed by these two points and the endpoints of $\xi_g$; this quadrilateral is $2\delta$--slim, so the Hausdorff distance between $\xi_g$ and the opposite side is at most $2\delta+M'$.  Finally, the opposite side is at Hausdorff distance at most $M'$ from the subpath $\beta_g'$ of $\beta_g$ connecting its endpoints by the Morse property of quasi-geodesics. 

In particular, the Hausdorff distance between $\gamma'$ and $\beta_g'$ is at most $2M'+8\delta$ and the endpoints of $\beta_g'$ are at distance at most $M'+4\delta$ from those of $\gamma'$.  See Figure \ref{fig:combthm}.
\begin{figure}
\centering
\def\svgwidth{5in}
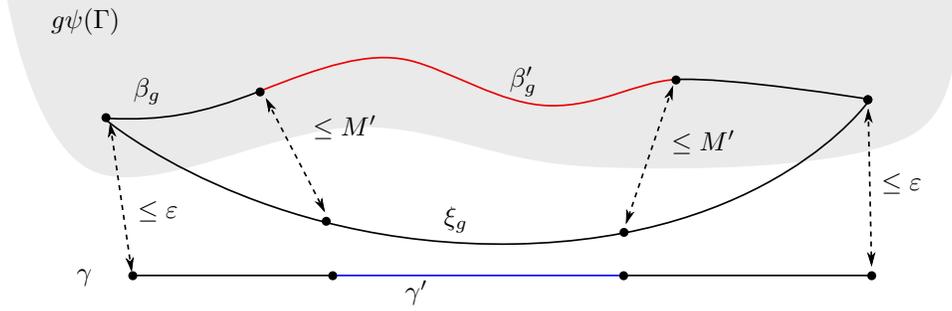
\caption{Finding the subpath $\beta_g'$ in the proof of Theorem \ref{thm:combination}.}
\label{fig:combthm}
\end{figure}

Each quasi-geodesic $\beta_g$ is a translate by $g$ of the image of a geodesic $\sigma_g$ in $\Gamma$ of the form 
\[
\sigma_g=\sigma_{g,1}\cdot \ldots \cdot \sigma_{g,p_g}
\]
as in \eqref{eqn:geodpath}.

We have 
\[
\frac{C}{\mu}-4M'-16\delta-1 \leq \ell_X(\beta_g') \leq \mu\left(\frac{C}{\mu}-1\right)+\mu = C.
\]
There may be the image of a wedge vertex in $\beta_g'$ for some $g$, but since the image of a geodesic between wedge vertices is a path of length at least $C$, there can be at most one such vertex in $\beta_g'$.   Suppose there is such a vertex $w$ in $\beta_g'$, dividing $\beta_g'$ into two $\tau$--quasi-geodesics.  Let $z$ be a closest point on $\gamma'$ to $w$.   Then $d_X(z,w)\leq 2M' + 8\delta$.  We also have that the initial points of $\gamma'$ and $\beta_g'$ are at distance at most $M' + 4\delta$. Therefore the initial subgeodesic $\gamma''$ of $\gamma'$ ending at $z$ and a geodesic $\beta_g''$ from the initial point of $\beta_g'$ to  $w$ are  at Hausdorff distance most $2M' + 10\delta$.  Moreover, the Hausdorff distance between the geodesic $\beta_g''$ and the corresponding subpath of $\beta_g'$ is at most $M$, since that subpath is a $\tau$--quasi-geodesic.
It follows that  $\gamma''$ is Hausdorff distance at most $2M' + M + 10\delta$ from the corresponding initial subpath of $\beta_g'$.  An analogous argument shows that the terminal subpath of $\gamma'$ starting at $z$ is at Hausdorff distance $2M' + M + 10\delta$ from the terminal subpath of $\beta_g'$ starting at $w$.

Therefore, for each $g$ as above there exists $\sigma_{g,j(g)}$ such that either the initial or terminal half of $\gamma'$ is contained in the $(2M' + M + 10\delta)$--neighborhood of $g\psi(\sigma_{g,j(g)})$, regardless of the number of images of wedge vertices appearing in $\beta_g'$.  Let $y_g$ be the initial vertex of the subpath of $g\psi(\sigma_{g,j(g)})$ contained in $\beta_g'$.  Then $y_g$ is either a wedge vertex or the initial vertex of $\beta_g'$, and lies in some translate of $Y_{j(g)}$.  (Recall that $Y_{i}=\psi(\Cay(H_i,T_i))$.)  The previous paragraph shows that 
\begin{equation}\label{eqn:distyg}
d_X(y_g,\gamma')\leq 2M'+M+10\delta.
\end{equation}   
If $\#\Delta>2NP$, then there exists a subcollection $\Delta'\subseteq\Delta$ with $\#\Delta'>NP$ such that the $(2M' + M + 10\delta)$--neighborhood of $g\psi(\sigma_{g,j(g)})$ contains (without loss of generality) the initial half of $\gamma'$ for all $gK\in \Delta'$.

However, $\ell_X(\gamma')=\frac{C}{\mu} - 2M' - 8\delta - 1 > 2D$ by \eqref{eqn:C}, and so, recalling that $\varepsilon' = 2M'+M+10\delta$, we have
\begin{equation}\label{eqn:Delta'D}
  \diam\left(\bigcap_{gK\in\Delta'}\mathcal N_{\varepsilon'}(g\psi(\sigma_{g,j(g)}))\right)>D.
\end{equation}
For each $g$, there is some $k_g\in K$ so that $g\psi(\sigma_{g,j(g)})$ is contained in $gk_gY_{j(g)}$.  Equation~\eqref{eqn:Delta'D} shows that the diameter of the intersection of the $\varepsilon'$--neighborhoods of the image of these Cayley graphs in $X$ is larger than $D$.  In order to contradict the acylindricity of the action of $G$ on $X$ along $\{H_1,\dots, H_k\}$, we must show that at least $N$ of these paths $g\psi(\sigma_{g,j_g})$ lie in  translates of these Cayley graphs corresponding to \emph{distinct} elements of $\bigsqcup G/H_i$.

 Since $\ell(\gamma')=\frac{C}{\mu}-2M'-8\delta-1$, it follows from \eqref{eqn:distyg} that for distinct $gK,g'K\in \Delta'$, we have 
\[
d_X(y_g,y_{g'})\leq 2(2M' + M + 10\delta) + \frac{C}{\mu} - 2M' - 8\delta - 1 =\frac{C}{\mu} + 2M + 2M' + 12\delta -1.
\]
By our choice of constant $P$, there must be a further subcollection $\Delta''\subseteq\Delta'$ with $\#\Delta''>N$ so that for distinct $gK,g'K\in \Delta''$, either $j(g)\neq j(g')$ or the vertices $y_g$ and $y_{g'}$, and therefore $g\psi(\sigma_{g,j(g)})$ and $g'\psi(\sigma_{g',j(g)})$,  lie in  translates of $Y_{j(g)}$ corresponding to distinct elements of $\bigsqcup G/H_i$.  In either case, this contradicts the acylindricity of the action of $G$ on $X$ along $\{H_1,\dots, H_k\}$.  We conclude $\#\Delta\le 2NP$, and so the action of $G$ on $X$ is acylindrical along $K$.  This completes the proof of the theorem.
\end{proof}

Collections of cyclic subgroups generated by WPD elements (see below) give an important special case of Theorem~\ref{thm:combination}.

\begin{definition} \label{def:WPD}
Let $G$ act on a hyperbolic space $X$, and let $g\in G$ act as a loxodromic isometry of $X$.  We say $g$ is a \emph{WPD element of $G$ (with respect to $X$)} if for every $x\in X$ and every $\varepsilon \geq 0$, there exists $N>0$ such that the set 
\[
\{\gamma\in G\mid d(x,\gamma x)\leq \varepsilon \textrm{ and } d_X(g^Nx,\gamma g^Nx)\leq \varepsilon\}
\] 
is finite.  The action of $G$ on $X$ is a \emph{WPD action} if $G$ is not virtually cyclic, $G$ contains at least one element which is loxodromic with respect to the action on $X$, and every loxodromic element is a WPD element.
\end{definition}

Suppose $g$ is a loxodromic WPD element with respect to an action of $G$ on a hyperbolic space.  Since $g$ is loxodromic, the subgroup $\langle g\rangle$  is quasi-isometrically embedded by the orbit map.  Moreover, the definition of a WPD element immediately implies that the action of $G$ is acylindrical along $\langle g\rangle$.
Recall that loxodromic isometries of a hyperbolic space are \emph{independent} if their fixed point sets at infinity are disjoint.  The following is clear from the definitions.
\begin{lemma}
If a group $G$ acts on a hyperbolic space $X$, then given any finite collection $h_1,\dots, h_k$ of independent loxodromic WPD elements with respect to this action, $(G,X,\{\langle h_1\rangle ,\dots, \langle h_k\rangle\})$ is an A/QI triple.
\end{lemma}   

The following corollary is then immediate from Theorem \ref{thm:combination}. 

\begin{corollary} \label{prop:WPDtoAQI}
Let $G$ act on a hyperbolic space $X$, and suppose $h_1,\dots, h_k$ are independent loxodromic WPD elements with respect to this action.  There exists a constant $D\geq 1$ such that for any $d_1,\dots, d_k\geq D$, if $H=\langle h_1^{d_1},\dots, h_k^{d_k}\rangle$, then $(G,X,H)$ is an A/QI triple.  In particular, every element in $H$ is a loxodromic WPD element with respect to the action of $G$ on $X$.
\end{corollary}

\begin{remark} \label{rem:notHE}
  Let $G$, $h_1,\ldots,h_k$ be as in the corollary.  For each $i$ let $E_i$ be the maximal elementary subgroup containing $h_i$.  Each $E_i$ is hyperbolically embedded in $G$, as is any subcollection of $\{E_1,\ldots,E_k\}$ obtained by choosing one subgroup per conjugacy class \cite[Theorem 6.8]{DGO}.  However, the subgroup $H$ in the conclusion of the corollary is typically not hyperbolically embedded, since it will not be almost malnormal if for some $i$ either  $E_i \neq \langle h_i\rangle$ or  $d_i>1$.
\end{remark}

\subsection{Finiteness of height}
We now turn our attention to the height of subgroups in A/QI triples.  We recall the definition.
\begin{definition}[Height]
  Let $G$ be a group and $H<G$.  If $H$ is finite the height of $H$ is $0$.  Otherwise the height of $H$ in $G$ is the largest $n$ so that there are distinct cosets $\{g_1 H,\ldots, g_n H\}$ so that the intersection $\bigcap g_i H g_i^{-1}$ is infinite.  If there is no largest $n$ we say the height is infinite.
\end{definition}

Quasi-convex subgroups of hyperbolic groups have finite height \cite{GMRS}, as do full
  relatively quasi-convex subgroups of relatively hyperbolic groups \cite{HruskaWise09}.   Finiteness of height for convex cocompact subgroups of mapping class groups and $\operatorname{Out}(\mathbb F_n)$ was established in \cite{DahmaniMj}, and, more generally, for stable  and even strongly quasi-convex subgroups of  finitely generated groups in \cite{AMST19, Tran}.  (Stable subgroups are discussed in more detail in Section \ref{sec:stability}.)

\begin{proposition}\label{prop:finiteheight}
  Suppose $(G,X,H)$ is an A/QI triple.  Then $H$ has finite height.
\end{proposition}
\begin{proof}
  Choose $\delta \ge 1$ so that $X$ is $\delta$--hyperbolic.  Then it is not hard to show that every loxodromic isometry of $X$ has a quasi-axis which is a $2\delta$--quasi-geodesic.
  
  Fix $x_0\in X$ and $Q>0$ so that $Hx_0$ is $Q$--quasi-convex and so that any $2\delta$--quasi-geodesic with endpoints in $Hx_0\cup \Lambda(Hx_0)$ is contained in a $Q$--neighborhood of $Hx_0$.

  Since $G\acts X$ is acylindrical along $H$,  Theorem \ref{thm:newacyl} yields constants $D, N>0$ so that any collection of distinct cosets $\{g_1H,\dots, g_k H\}$ satisfying 
  \[
  \diam\left(\bigcap_{i=1}^k \mathcal N_Q(g_i Hx_0)\right)>D
\] 
has cardinality at most $N$.

  Suppose that $\{g_iH\}_{i\in I}$ is a collection of distinct cosets of $G$ so that the corresponding intersection of conjugates $K=\bigcap_{i\in I}g_iHg_i^{-1}$ is infinite.  We show $H$ has finite height by giving a bound on the cardinality of $I$ in terms of the quantities already specified.

  The intersection $K$ is an infinite subgroup of the hyperbolic group $H$, so it contains an infinite order element $w$ (see~\cite[Ch. 8, Cor. 36]{GdlH}).  Since $H$ is quasi-isometrically embedded by the action, $w$ acts loxodromically on $X$. Replacing $w$ by a power if necessary,
 $w$ has a $2\delta$--quasi-geodesic axis $\gamma$ with endpoints in $\Lambda(g_iH)$ for each $i$.

 By our assumptions on $Q$, the axis $\gamma$ is contained in the $Q$--neighborhood of $g_iHx_0$ for each $i$.  In particular, this implies that 
 \[
\diam\left( \bigcap_{i\in I} \mathcal N_{Q}(g_iHx_0)\right)>D.
 \]
By acylindricity, we therefore have $\#I\leq N$.
\end{proof}

\section{Stability}\label{sec:stability}
In this section we prove Theorem \ref{thm:stable}.  We first recall the definition of a \emph{stable subgroup} and then restate the theorem for convenience.

\begin{definition}\label{def:stable}
  Let $G$ be a finitely generated group.  A finitely generated subgroup $H$  of $G$ is \emph{stable} if it is undistorted and for some (equivalently, any) finite generating set $S$ of $G$ and for every $\tau \ge 1$, there exists $L=L(S,\tau)$ such that whenever $\gamma_1,\gamma_2$ are $\tau$--quasi-geodesics in $\Cay(G,S)$ with the same endpoints in $H$, the Hausdorff distance between $\gamma_1$ and $\gamma_2$ is at most $L$.
\end{definition}

\stable*
Known examples of stable subgroups include convex cocompact subgroups of mapping class groups \cite{DurhamTaylor15} and $\OutFn$ \cite{AougabDurhamTaylor}, cyclic subgroups generated by a loxodromic WPD element in an acylindrically hyperbolic group \cite{Sisto16}, subgroups of relatively hyperbolic groups which are quasi-isometrically embedded into the coned-off Cayley graph by the orbit map \cite{AougabDurhamTaylor}, and subgroups of hierarchically hyperbolic groups which are quasi-isometrically embedded into the top-level curve by the orbit map \cite{ABD}.

\begin{definition}\label{def:sqc}
Let $G$ be a finitely generated group.  A finitely generated subgroup $H$ of $G$ is \emph{strongly quasi-convex} if for some (equivalently, any) finite generating set $S$ of $G$ and for every $\tau\ge 1$, there exists $L=L(S,\tau)$ such that  every $\tau$--quasi-geodesic in $\Cay(G,S)$ with endpoints in $H$ is contained in the $L$--neighborhood of $H$.
\end{definition}

We will use the following observation of Tran to prove stability.
\begin{theorem} [{\cite[Theorem~4.8]{Tran}}] \label{obs:tran} A subgroup of a finitely generated group is stable if and only if it is hyperbolic and strongly quasi-convex in $G$.
\end{theorem}

In \cite{Sisto16}, Sisto investigates quasi-convexity properties of hyperbolically embedded subgroups of acylindrically hyperbolic groups and shows that they are strongly quasi-convex.  The arguments in our proof of Theorem \ref{thm:stable} are based on those of Sisto.  However, we work in a slightly different setting.  
Sisto works under the assumption that $X$ is a Cayley graph of $G$ with respect to an infinite generating set.  By a Milnor--Schwartz argument, if $G$ acts coboundedly on $X$, then this is equivalent to our situation where $X$ is a metric space.  However, we do not assume that the action of $G$ on $X$ is cobounded.  Also, Sisto does not assume that $X$ is a hyperbolic metric space, rather assuming that $X$ is hyperbolic relative to the translates of $\pi(H)$.  This assumption is neither weaker nor stronger than ours---our space $X$ is assumed to be hyperbolic,  but might not be hyperbolic \emph{relative} to the translates of $\pi(H)$, since these could have unbounded coarse intersection.

\subsection{An axiomatic approach to Sisto's proof}
Though Sisto's setting differs from ours, there are a few features in common, and these are all which are needed to execute Sisto's strategy.  For this subsection we fix a geodesic space $X$, an isometric action $G\acts X$, and a subgroup $H<G$, so that both $G$ and $H$ are finitely generated.  We fix a basepoint $x_0\in X$, and let $\pi\from G\to X$ be the orbit map $\pi(g) = gx_0$.  The following axioms will be assumed for the rest of the subsection.
\begin{axiomlist}
  \item\label{ax:qie} The restriction of $\pi$ to $H$ is a quasi-isometric embedding.
  \item\label{ax:cl} There is a nearest-point projection map $\rho\from X\to \pi(H)$ satisfying the following:  For all $a_0>0$, there is an $a>0$ so that if $d_X(x,y)\le a_0$, then $d_X(\rho(x),\rho(y))\le a$.
  \item\label{ax:contracting}
    There is a constant $C>0$ so that
    $ d_X(\rho(x),\rho(y))\ge C $ implies \[ d_X(x,y)\ge \max\{ d_X(x,\pi(H)), d_X(y,\pi(H))\}.\]
  \item\label{ax:wa} $G\acts X$ is weakly acylindrical along $\pi(H)$.
\end{axiomlist}

\begin{remark}\label{rem:axiomshold}
  For Sisto, $X$ is a Cayley graph of $G$ which is not necessarily hyperbolic, but is hyperbolic \emph{relative to} a hyperbolically embedded family $\mc{H}$ including the subgroup $H$.
  
  Dahmani--Guirardel--Osin show (see the argument for \cite[Theorem 4.31]{DGO}) that in this case $H$ is a Lipschitz quasi-retract of $X$; in particular $H$ is quasi-isometrically embedded in $X$ (Axiom~\ref{ax:qie}). 
  Also in Sisto's setting, Axioms~\ref{ax:cl} and~\ref{ax:contracting} follow from~\cite[Lemma 2.3, parts (2) and (3)]{Sisto16}.  Axiom~\ref{ax:wa} is~\cite[Lemma 3.2]{Sisto16}.

  In our setting of an A/QI triple $(G,X,H)$, $H$ is not necessarily part of a hyperbolically embedded collection, but rather the axioms~\ref{ax:qie} and~\ref{ax:wa} are part of the definition.  Axioms~\ref{ax:cl} and~\ref{ax:contracting} are straightforward exercises in Gromov hyperbolic geometry, and use only that $\pi(H)$ is quasi-convex in the hyperbolic space $X$.
\end{remark}

Since $H$ and $G$ are finitely generated, we may fix finite generating sets $S\subset T$ of $H$ and $G$ respectively, and  consider the Cayley graph $\Cay(H,S)$ as a subset of the Cayley graph $\Cay(G,T)$ in the natural way.   We  use $d_G$ and $d_X$ to denote the metrics on $\Cay(G,T)$ and $X$, respectively, and $\ell_G(\alpha)$ and $\ell_X(\beta)$ to denote the length of a path $\alpha$ in $G$ (more precisely, in $\Cay(G,T)$) and a path $\beta$ in $X$, respectively.

\medskip

\putinbox{For the rest of the section we fix a constant $\kappa\ge 1$ so that $d_X(x_0,gx_0)\le \kappa$ for every $g\in T$.}

\medskip

We  use $B^X_r(x)$ to denote the ball in $X$ about $x\in X$ with radius $r$ (that is, $r$ is measured with respect to $d_X$).  We use $N^G_K(Y)$ and $N^X_K(Z)$ to denote the closed $K$--neighborhood in $G$ of $Y\subset \Cay(G,T)$ and the closed $K$--neighborhood in $X$ of $Z\subset X$, respectively.

\begin{definition}
  Two subsets $A,B$ of a metric space $Y$ are said to be \emph{geometrically separated} if, for all $D>0$, the diameter of $N_D(A)\cap B$ is finite. 
\end{definition}

The first lemma is \cite[Lemma~3.3]{Sisto16}, slightly corrected.   It shows that preimages of balls of $X$ which are far apart are geometrically separated in $\Gamma(G,T)$, and only uses Axiom~\ref{ax:wa}.

\begin{lemma}[{\cite[Lemma~3.3]{Sisto16}}]\label{lem:sisto3.3}
  Let $G$ act on the geodesic space $X$, weakly acylindrically along the subspace $Y$.  Let $x_0\in X$, and let $\pi$ denote the orbit map $g\mapsto gx_0$.  For each $r\ge 0$, there is an $R_1\ge 0$ so that if $x,y\in Y$ satisfy $d_X(x,y)\ge R_1$, then $\pi^{-1}(B_r^X(x))$ and $\pi^{-1}(B_r^X(y))$ are geometrically separated.
\end{lemma}
\begin{proof}
  Let $r\ge 0$.  By Axiom~\ref{ax:wa}, $G$ acts weakly acylindrically along $Y$.   In particular there is some $R_1$ so that for any $x,y\in Y$ with $d_X(x,y)\ge R_1$ we have
  \begin{equation*}
    \#\{g\mid d_X(gx,x),d_X(gy,y)\le 2r\}<\infty.
  \end{equation*}
    We therefore fix some $x,y\in Y$ so that $d_X(x,y)\ge R_1$, and let $A = \pi^{-1}(B_r^X(x))$ and $B = \pi^{-1}(B_r^X(y))$.
  \begin{claim*}
    For a fixed $g\in G$, 
    $\#\{a\in A\mid ag\in B\}<\infty$.
  \end{claim*}

  Given the claim, we argue by contradiction.  Suppose for some $D>0$, the intersection $N_D^G(A)\cap B$ is unbounded.  There are thus infinitely many pairs of vertices $(a_i,b_i)$ with $a_i\in A$, $b_i\in B$, and $d_G(a_i,b_i)\le D$.  Thus all the elements $a_i^{-1}b_i$ lie in $B_D^G(1)$, which is finite.  Passing to a subsequence, the difference $a_i^{-1}b_i$ is a constant element $g$.  For this $g$, the set $\{a\in A\mid ag\in B\}$ is infinite, contradicting the claim.

  \begin{proof}[Proof of Claim]
    Let $\Omega = \{a\in A\mid ag \in B\}$, and fix some $a_0\in \Omega$.  We want to show $\Omega$ is finite.

    For any $a\in \Omega$, we have \[d_X(x,aa_0^{-1}x)\le d_X(x,ax_0)+d_X(ax_0,aa_0^{-1}x) = d_X(x,ax_0)+d_X(a_0x_0,x)\le 2r.\]  Similarly we have \[d_X(y,aa_0^{-1}y)\le d_X(y,agx_0)+d_X(agx_0,aa_0^{-1}y) = d_X(y,agx_0)+d_X(a_0gx_0,y)
    \le  2r.\]  Since $d_X(x,y)\ge R_1 $ and $x,y\in Y$, the number of possibilities for $aa_0^{-1}$ (and hence for $a$) is finite.
  \end{proof}
\end{proof}

The next lemma uses both Axiom~\ref{ax:wa} and Axiom~\ref{ax:qie} (in the form of the fact that $\pi|_H$ is proper).
\begin{lemma}[{\cite[Lemma~3.4]{Sisto16}}]\label{lem:geomsep}
    
    For each $r\ge 0$, there is an $R>0$ and a function $B\from \bR^+\to \bR^+$ so that for any $x,y\in H$ with $d_X(\pi(x),\pi(y))\ge R$ we have, for any $D\ge 0$,
    \begin{equation*}
      N^G_D(\pi^{-1}(B^X_r(\pi(x))))\cap \pi^{-1}(B^X_r(\pi(y)))\subseteq N^G_{B(D)}(H).
    \end{equation*}
  \end{lemma}

\begin{proof}
  
  We note that the orbit map $\pi\from G\to X$ is $\kappa$--Lipschitz.
  Fix some $r\ge 0$.  By equivariance it suffices to consider $x = 1$, $y = h\in H$.
  
  By Lemma~\ref{lem:sisto3.3}, there is an $R=R_1(2r)\ge 0$ so that if $h\in H$ satisfies $d_X(x_0,\pi(h))\ge R$, then the sets $\pi^{-1}(B_{2r}^X(x_0))$ and  $\pi^{-1}(B_{2r}^X(\pi(h)))$ are geometrically separated.

  Now let $D>0$.  If $d_X(x_0,\pi(h))>\kappa D+2r$, then the sets $N_D^G(\pi^{-1}(B_{r}^X(x_0)))$ and $\pi^{-1}(B_{r}^X(\pi(h)))$ are disjoint.  Since $\pi|_H$ is proper (Axiom~\ref{ax:qie}), the set \[\Psi = \{h\in H\mid R\le d_X(x_0,\pi(h))\le \kappa D + 2r\}\] is finite.  For each $h\in\Psi$, Lemma~\ref{lem:sisto3.3} tells us that the diameter
  \[ d(h) = \diam N_D^G(\pi^{-1}(B_{2r}^X(x_0)))\cap \pi^{-1}(B_{2r}^X(\pi(h))) \] is finite.  We let $B(D)= \max\{d(h)\mid h\in \Psi\}$.
\end{proof}

Let $\alpha$ be a tame $\tau$--quasi-geodesic edge-path in the Cayley graph of $G$ with respect to $T$.  The path $\alpha$ passes through a sequence of group elements $a_1,\ldots,a_n$.  By $\pi(\alpha)$, we mean the discrete path $\pi(a_1),\ldots,\pi(a_n)$.  Given a discrete path $\gamma=x_0,x_1,\dots,x_n$ in $X$, recall that we define the length of $\gamma$ to be 
\[\ell^0(\gamma):=\sum_{i=0}^{n-1}d_X(x_i,x_{i+1}).\]
If all the summands $d_X(x_i,x_{i+1})$ are at most $r$, we say that $\gamma$ is an \emph{$r$--coarse path}.  Since $\kappa = \max\{d_X(x_0,tx_0)\mid t\in T\}$, the discrete path $\pi(\alpha)$ is a $\kappa$--coarse path.  In particular we have
\begin{equation}\label{eq:lengthbound}
  \ell_G(\alpha)\geq \ell^0(\pi(\alpha))/\kappa.
\end{equation}

 By choosing geodesics in $X$ between $x_0$ and $sx_0$ for each $s\in S$, we extend the map $\pi|_H$ to an $H$--equivariant map 
\[ \hat\pi\from \Cay(H,S)\to X,\]
The image $\Gamma = \hat\pi(\Cay(H,S))$ is an $H$--cocompact, embedded,  connected subset which contains the discrete set $\pi(H)$.
Fix $\lambda\ge 1$  so that $\hat\pi$ is a $\lambda$--quasi-isometric embedding.

\medskip
\putinbox{The map $\hat\pi$ and the constant $\lambda$ will be fixed for the rest of the section.}
\medskip

For $h\in H$ and $0\leq r_1\leq r_2$, we define the following subspace of $\Gamma$, which we refer to as an \emph{annulus}:  \[A(h,r_1,r_2)=\Gamma\cap \left(B^X_{r_2}(hx_0)\setminus\mathring B^X_{r_1}(hx_0)\right).\]  Since $\Gamma$ is connected, contains $x_0$, and has infinite diameter, $A(h,r_1,r_2)$ is always non-empty.

  \begin{proposition}[{\cite[Proposition~4.1]{Sisto16}}] \label{claim:stable}
    For any $L\geq 0$ there exists a  $K\geq 0$ such that the following holds.  For any $h,h'\in H$ with $d_G(h,h')>K$ and any $0<r_1<r_2<d_X(hx_0,h'x_0)$, if $\alpha$ is a path in $G$ from $h$ to $h'$ and $\alpha\cap N^G_K(\hat\pi^{-1}(A(h,r_1,r_2)))=\emptyset$, then $\ell_G(\alpha)\geq L(r_2-r_1)$. 
\end{proposition}
\begin{proof}
  We may as well suppose $L \ge 1$.  Recall that $\kappa\ge 1$ is a Lipschitz constant for the orbit map $\pi$.  
  By equivariance, we may consider just paths $\alpha$ joining $1$ to $h$ for $h\in H$.  
  The map $\pi$ is only defined on group elements, so $\pi(\alpha)$ is a $\kappa$--coarse path in $X$.
  We will obtain our lower bound on $\ell_G(\alpha)$ by analyzing how $\pi(\alpha)$ interacts with $X$--neighborhoods of regularly spaced sub-annuli of $A(h,r_1,r_2)$.  If it misses enough of these, then $\pi(\alpha)$ is already quite long, and thus so is $\alpha$.  If on the other hand it meets many such neighborhoods, we will see that it meets many pairs of such neighborhoods at controlled distances from each other.  When this happens we will be able to use Lemma~\ref{lem:geomsep} to bound $\ell_G(\alpha)$ from below.

Recall that $\rho\from X\to \pi(H)$ is a nearest point projection map.  By Axiom~\ref{ax:cl}, there is an $a$ so that $\rho$ takes every $\kappa$--coarse path in $X$ to an $a$--coarse path of points in $\pi(H)$.

We now fix several more constants.  Some of these will only be used much later in the proof.
Let $D_1=C+a$, where $C$ is the constant from Axiom~\ref{ax:contracting}.  
Let $C_2\geq 4000LD_1\kappa+a$.  
Also fix an integer $R\geq \kappa$ so that Lemma~\ref{lem:geomsep} holds with $r=C_2 + \kappa/2$.  We set $D_2 = 100C_2RL$, and let $B = B(D_2)$,  where $B(\cdot)$ is the function from Lemma~\ref{lem:geomsep}. 
Finally, let
\begin{equation*}
  B' = \lambda\left(\kappa B+ C_2 \right)+\lambda^2 
\end{equation*}
and let
\begin{equation*}
  K = \max\{1000 C_2 R L, B+B'\}.
\end{equation*}

We now assume that $h$ satisfies $d_G(1,h)>K$ and choose $r_1<r_2$ between $0$ and $d_X(x_0,hx_0)$.  Let $\alpha$ be a path in $\Cay(G,T)$ satisfying 
\[\alpha\cap N^G_K(\hat\pi^{-1}(A(1,r_1,r_2)))=\emptyset.\]
We must bound the length of $\alpha$ from below.

  There are two cases, depending on whether $r_2-r_1 \geq 1000 C_2R$ or not.  In the case $r_2-r_1 < 1000C_2R$, we argue as follows.
By assumption, $d_G(1,h)>K$, and so $\ell_G(\alpha)> K$.  Since $K\geq 1000C_2RL$, it follows that $\ell_G(\alpha)\geq 1000C_2RL> (r_2-r_1)L$, completing the proof in this case.

Now suppose $r_2-r_1\geq 1000C_2R$.

Let $n =\lfloor (r_2-r_1)/(10C_2)\rfloor$.  Let $r_1 = d_0,d_1,\ldots,d_n = r_2$ be evenly spaced real numbers, and notice that $d_{i+1}-d_i \ge 10 C_2$.  Consider the following annuli and neighborhoods of annuli:
\[A_i = A(1,d_i,d_i+C_2)\quad\mbox{and}\quad T_i=N^X_{C_2}(A_i).\] 

\setcounter{claim}{0}
\begin{claim}\label{cl:1}
Suppose the number of indices $i$ such that $\pi(\alpha)\cap T_i=\emptyset$ is greater than $(r_2-r_1)/(1000C_2)$.  Then $\ell_G(\alpha)\geq L(r_2-r_1)$.
\end{claim}
\begin{proof}[Proof of claim]
  At least $(r_2-r_i)/(1000 C_2)$ of the indices satisfying $\pi(\alpha)\cap T_i=\emptyset$ are strictly less than $n$.
  Fixing one such index $i$, we will find a subpath $\beta_i$ of $\pi(\alpha)$ with some definite length.  These will be disjoint subpaths and the sum of their lengths will give the desired lower bound on $\ell_G(\alpha)$.

  For a fixed $i$, let $x_i$ be the last point on $\pi(\alpha)$ so that $d_X(x_0, \rho(x_i))< d_i$, and let $y_i$ be the first point after $x_i$ on $\pi(\alpha)$ so that $d_X(x_0,\rho(y_i))> d_i + C_2$.  Note that such a $y_i$ always exists; otherwise we would have $d_X(x_0,hx_0) \le d_i + C_2$ which could only happen for $i=n$,  contradicting our assumption on $i$. 
  Let $\beta_i=\pi(\alpha)|_{[x_i,y_i]}$.  
  We will bound the length of this subpath from below.

\begin{subclaim*}
There are points $q_1,\dots, q_m$ on $\beta_i$ satisfying:
\begin{enumerate}[label=(\alph*),ref=(\alph*)]
\item\label{q:many} $m\geq C_2/D_1$
\item\label{q:spaced} $d_X(\rho(q_j),\rho(q_{j+1}))\ge C$
\item\label{q:nearsphere} $\rho(q_j)\in A_i$.
\end{enumerate}
\end{subclaim*}
\begin{proof}
   Since $\rho(\pi(\alpha))$ is an $a$--coarse path, the set of distances $\{d_X(x_0,\rho(q))\mid q\in \beta_i\}$ is $\frac{a}{2}$--dense in the interval $[d_X(x_0,\rho(x_i)),d_X(x_0,\rho(y_i))] \supset [d_i,d_i + C_2]$.
   In particular, there is some $q_1$ in $\beta_i$ so that $d_X(x_0,\rho(q_1))\in [d_i,d_i+a]$.  If we have chosen $q_j$, we can proceed inductively.  The set of distances $\{d_X(x_0,\rho(q))\mid q\in \beta_i|_{[q_j,y_i]}\}$ is $\frac{a}{2}$--dense in the interval $[d_X(x_0,\rho(q_j)),d_i+C_2]$.  As long as this interval contains $[d_X(x_0,\rho(q_j))+C, d_X(x_0,\rho(q_j)) + C+a]$, we can find $q_{j+1}\in \beta_i$ so $d_X(x_0,\rho(q_{j+1}))$ lies in that interval.  The triangle inequality ensures item~\ref{q:spaced} holds.  Since $d_X(x_0,\rho(q_{j+1}))\in [d_i,d_i + C_2]$, item~\ref{q:nearsphere} also holds for $q_{j+1}$.
   If we fail to find such a $q_{j+1}$ this can only be because $j+1> C_2/D_1$.  (Recall $D_1 = C+a$.)
\end{proof}   

It follows from~\ref{q:spaced} and Axiom~\ref{ax:contracting} that $d_X(q_j,q_{j+1})\geq d_X(q_j,\pi(H))$.  Since $q_j\not\in T_i$, \ref{q:nearsphere} implies that $d_X(q_j,\pi(H))\geq C_2/2$.  Combining this with \ref{q:many}, we conclude that \[\ell^0(\beta_i)\geq \sum_{j=1}^{m-1} d_X(q_j,q_{j+1})\geq \frac{C_2}{2}\left(\frac{C_2}{D_1}-1\right)\geq \frac{C_2^2}{4D_1}\geq 1000C_2L\kappa.\] 

Since $d_{i+1}-d_i \ge 10 C_2$, these subpaths $\beta_i$ must be disjoint for distinct $i$.
In particular, if the set of $i<n$ such that $\pi(\alpha)\cap T_i=\emptyset$ is at least $(r_2-r_1)/(1000C_2)$, then \[\ell^0(\pi(\alpha))\geq \frac{r_2-r_1}{1000C_2}\cdot 1000C_2L\kappa\geq L\kappa(r_2-r_1).\]
Applying~\eqref{eq:lengthbound} we obtain $\ell_G(\alpha)\geq L(r_2-r_1)$, as desired.
 \end{proof}

 Let $I_0\subset \{1,\ldots n\}$ be the set of indices so that $\pi(\alpha)\cap T_i\neq\emptyset$.  Claim~\ref{cl:1} implies that $\# I_0 \ge n-(r_2-r_1)/(1000C_2)\geq (1-2/100)n$.  Subdividing the integers modulo $R$ there must exist $0\leq s\leq R-1$ such that the residue class $s+R\bZ$ meets $I_0$ in a set of cardinality at least $(1-\frac{2}{100})\frac{n}{R}$.  A counting argument shows that there is a set $I\subset I_0\cap (s+R\bZ)$ of cardinality at least $(1-\frac{4}{100})\frac{n}{R}\geq \frac{r_2-r_1}{100C_2R}$ so that for every $i\in I$ we have both $\pi(\alpha)\cap T_i\neq \emptyset$ and $\pi(\alpha)\cap T_{i+R}\neq \emptyset$. 

  For each $i\in I\cup (R+I)$, choose a point $a_i\in \pi^{-1}(T_i)\cap \alpha$.  For distinct $i,i'\in I$ we note that the subpaths $\alpha|_{[a_i,a_{i+R}]}$ and $\alpha|_{[a_{i'},a_{i'+R}]}$ contribute separately to $\ell_G(\alpha)$.

\begin{claim}\label{cl:2}
  For each $i\in I$, we have $d_G(a_i,a_{i+R})\ge 100 C_2 R L$.
\end{claim}

Assuming Claim~\ref{cl:2}, as there are at least $(r_2-r_1)/(100C_2R)$ such pairs of points $a_j,a_{j+R}$ and $\ell_G(\alpha)\geq \sum_j d_G( a_j,a_{j+R})$, we  have $\ell_G(\alpha)\geq L(r_2-r_1)$, as desired.  Thus it remains to prove the claim.

\begin{proof}[Proof of claim]
  We assume for a contradiction that $d_G(a_i,a_{i+R}) < 100 C_2 R L$ for some $i\in I$.  We will show that this forces the path $\alpha$ to enter the $K$--neighborhood of $\hat\pi^{-1}(A(h,r_1,r_2))$, contradicting our basic assumption on $\alpha$.

  For $k\in\{i,i+R\}$, let $p_k = \pi(a_k)\in T_k$.
There is some point $b_k \in A_k$ so that $d_X(p_k,b_k)\le C_2$, and some $c_k\in \pi(H)$ with $d_X(b_k,c_k)\le \kappa/2$.  We have $d_X(p_k,c_k)\le C_2 + \kappa/2$.  Moreover
\begin{equation}
  \label{eq:cjdist}
  d_X(c_i,c_{i+R}) \ge 10 RC_2 - C_2 - \kappa > R,
\end{equation}
so we may apply Lemma~\ref{lem:geomsep} with $c_{i+R}$ and $c_i$ taking the place of  $\pi(x)$ and $\pi(y)$, respectively.
Accordingly, let $x\in H\cap \pi^{-1}(c_{i+R})$ and let $y\in H\cap \pi^{-1}(c_i)$.  We have $a_i\in \pi^{-1}(B_{C_2+\kappa/2}^X(\pi(y)))$, since $d_X(p_i,c_i)\le C_2 + \kappa/2$.  Recalling that $D_2 = 100C_2RL$, we also have
$a_i\in N_{D_2}(\pi^{-1}(B_{C_2 + \kappa/2}^X(\pi(x))))$ since we have assumed $d_G(a_i,a_{i+R}) < D_2$.  Now Lemma~\ref{lem:geomsep} implies that there is some $h''\in H$ within $B = B(D_2)$ of $a_i$. (In $X$ we have $d_X(\pi(h''),p_i) \le \kappa B$.)

Now we note that $b_i\in A_i \subset A(1,r_1,r_2)$.
Thus
\[ d_G(a_i,\hat\pi^{-1}(A(1,r_1,r_2)))\le d_G(a_i,\hat\pi^{-1}(b_i)).\]
We will see this is less than $K$.  Indeed
\begin{align*}
  d_G(a_i, \hat\pi^{-1}(b_i)) & \le d_G(a_i,h'') + d_G(h'',\hat\pi^{-1}(b_i))\\
  & \le B + d_G(h'',\hat\pi^{-1}(b_i)).
\end{align*}
Now
\begin{align*}
  d_X(\pi(h''),b_i)& \le d_X(\pi(h''),p_i) + d_X(p_i,b_i)\\
                   & \le \kappa B + C_2,
\end{align*}
and since $\hat\pi$ is a $\lambda$--quasi-isometric embedding, this implies
\[ d_G(h'',\hat\pi^{-1}(b_i))\le \lambda(\kappa B + C_2) + \lambda^2  = B' .\]
We thus have $d_G(a_i, \hat\pi^{-1}(b_i))\le B + B' = K$, as desired.  This contradiction proves the claim.
\end{proof}
\end{proof}

We now use Proposition~\ref{claim:stable} to prove strong quasi-convexity (Definition~\ref{def:sqc}).  By Tran's Theorem~\ref{obs:tran}, a subgroup is stable if and only if it is both hyperbolic and strongly quasi-convex.

\begin{theorem}\label{thm:axioms_imply_sqc}
  Let $G$ be a finitely generated group acting on a geodesic space $X$, and let $H<G$ be finitely generated.  If $(G,X,H)$ satisfy the axioms \ref{ax:qie}--\ref{ax:wa}, then $H$ is strongly quasi-convex in $G$.
\end{theorem}
\begin{proof}
  In order to show strong quasi-convexity it suffices to consider tame quasi-geodesics (see the discussion after Definition \ref{def:quasi-geodesic}); in a graph, these tame quasi-geodesics can be taken to be edge-paths.
  So we assume $\alpha$ is a tame $\tau$--quasi-geodesic edge-path in $\Cay(G, T)$ from  $h_1\in H$ to $h_2\in H$.  We will find a neighborhood of $H$ containing $\alpha$ and only depending on the value of $\tau$.  Recall that $\hat\pi$ is a $\lambda$--quasi-isometric embedding.

Fix $L=\lambda(\tau+1)$, and let $K$ be as in Proposition~\ref{claim:stable} for this choice of $L$.  Suppose that $\alpha$ is not contained in $N^G_{K}(H)$.  Let $x,y$ be vertices of  $\alpha$  so that the subpath $\alpha'$ of $\alpha$ from $x$ to $y$ intersects $N^G_K(H)$ only at its endpoints.  Since $\alpha$ is a tame $\tau$--quasi-geodesic, it follows that
\begin{equation}\label{eqn:qgeo}
  \ell_G(\alpha')\leq \tau d_G(x,y)+\tau.
\end{equation}
  We will bound $\ell_G(\alpha')$ by a constant $M$  depending only on $\lambda$ and $\tau$, which will imply that $\alpha\subset N^G_{K+M}(H)$, as desired.  Appealing to \eqref{eqn:qgeo}, it suffices to bound $d_G(x,y)$.  We claim:
  \begin{equation}
    \label{eq:dxybound}
    d_G(x,y) \le \max\left\{
      4\lambda^2K + 3\lambda^2 + 2K,\
      \tau + 4L\lambda K + 3 L \lambda + 2K + 2L\lambda^{-1}K
    \right\}.
  \end{equation}

  Consider the path $\beta'$ obtained by concatenating (in the appropriate order) $\alpha'$ with geodesics of length at most $K$ connecting $x$ and $y$ to points $h,h'\in H$, respectively.

  We divide into two cases, depending on whether or not
\begin{equation}
  \label{eq:dichot}
  d_G(h,h') > 4\lambda^2K + 3\lambda^2\,.
\end{equation}
If \eqref{eq:dichot} does \emph{not} hold, then since $x,y$ are distance at most $K$ from $h,h'$ respectively,
\begin{equation*}
  d_G(x,y) \le 4\lambda^2K + 3\lambda^2 + 2K\,,
\end{equation*}
establishing \eqref{eq:dxybound} in this case.

So we suppose that \eqref{eq:dichot} holds.
Set $r_1=2\lambda K + \lambda$
and $r_2=d_X(\hat \pi(h),\hat \pi(h'))-(2\lambda K + \lambda)$.
From \eqref{eq:dichot} and the fact that $\hat \pi$ is a $\lambda$--quasi-isometric embedding, we obtain $r_1<r_2$, so $A = A(h,r_1,r_2)$ is nonempty.
 Moreover, $r_1$ and $r_2$ are chosen so that $d_G(\{h,h'\},\hat\pi^{-1}(A))\ge 2K$.  In particular the initial and terminal segments of $\beta'$ connecting $\alpha'$ to $H$ lie outside $N_K^G(\hat\pi^{-1}(A))$.   
 Since $\hat\pi^{-1}(A)$ is contained entirely within $\Cay(H,S)$, the path $\alpha'$ also lies outside $N_K^G(\hat\pi^{-1}(A))$.  Thus we have
 \begin{equation*}\beta'\cap N_K^G(\hat\pi^{-1}(A(h,r_1,r_2)))=\emptyset\,.\end{equation*}

Now Proposition~\ref{claim:stable} yields $\ell_G(\beta') \geq L(r_1-r_2)$.  We have
\begin{align*}
  \ell_G(\alpha')& \geq L(r_2-r_1)-2K\\
                 & \geq L\left( d_X(\hat \pi(h),\hat \pi(h'))-(4\lambda K + 2 \lambda)\right)-2K\\
                 & \geq L (\lambda^{-1} d_G(h,h') - \lambda)- (4L\lambda K + 2 L \lambda + 2K)\\
                 & \geq (\tau+1) d_G(x,y) - (4L\lambda K + 3 L \lambda + 2K + 2L\lambda^{-1}K)\,.
\end{align*}
Now \eqref{eqn:qgeo} tells us that $\ell_G(\alpha')$ is at most $\tau d_G(x,y) + \tau$, so (subtracting $\tau d_G(x,y)$ from both sides and rearranging),
\begin{equation*}
  d_G(x,y) \leq \tau + 4L\lambda K + 3 L \lambda + 2K + 2L\lambda^{-1}K\,,
\end{equation*}
verifying \eqref{eq:dxybound} in this case.
\end{proof}

\subsection{Application to A/QI triples}
Returning to our main setting, we have already noted in Remark~\ref{rem:axiomshold} that if $(G,X,H)$ is an A/QI triple, all the axioms~\ref{ax:qie}--\ref{ax:wa} hold.
From this we easily derive Theorem~\ref{thm:stable}.  

\begin{proof}[Proof of Theorem~\ref{thm:stable}]
 Since $H$ quasi-isometrically embeds into the hyperbolic space $X$, it must be hyperbolic.   By Tran's Theorem~\ref{obs:tran}, it therefore suffices to show that $H$ is strongly quasi-convex, which is the content of Theorem~\ref{thm:axioms_imply_sqc}.
\end{proof}

The work of Antol\'in--Mj--Sisto--Taylor \cite[Theorem 1.1]{AMST19} (cf. \cite[Theorem 1.2(3)]{Tran}) gives the following corollary:
\begin{corollary}
  Let $(G,X,H)$ be an A/QI triple, and suppose that $G$ is finitely generated.  Then $H$ has finite height, finite width, and bounded packing.
\end{corollary}

\section{The boundary of a cone-off}\label{sec:coneoff}
Let $X$ be a graph, and let $\mc{H}$ be a collection of subgraphs.  The \emph{cone-off of $X$ with respect to $\mc{H}$} is the space $\hat{X}_{\mc{H}}$  obtained from $X$ by adding an edge between each pair of distinct $x,y$, so that $\{x,y\}\subset Y^{(0)}$ for $Y\in \mc{H}$.  These new edges are called \emph{electric edges}.
If there is no ambiguity about $\mc{H}$, we may use $\hat{X}$ to denote the cone-off.

\begin{convention}
In this section we will restrict to continuous quasi-geodesics which are tame; see the discussion after Definition \ref{def:quasi-geodesic}.
\end{convention}

\begin{convention}  In this section, we consider two hyperbolic spaces, $X$ and $\hat X = \hat{X}_{\mc{H}}$.  Fix a basepoint $x_0\in X$; by definition, $x_0\in\hat X$, as well.  We will frequently calculate Gromov products in each space, and so to avoid confusion, we will use $(\cdot\mid\cdot)_{x_0}$ to denote the Gromov product in $X$ and $\langle \cdot\mid\cdot\rangle_{x_0}$ to denote the Gromov product in $\hat X$.  
\end{convention}

There are various other possibilities for a definition of a cone-off of $X$.  It is clear, for example, that if $\hat{X}$ is the cone-off of $X$ with respect to $\mc{H}$, then it is quasi-isometric to the space obtained by adjoining, for each $Y\in \mc{H}$, a new vertex $v_Y$ connected by an edge to every vertex of $Y$.  When $X$ is the Cayley graph of a group, this construction is due to Farb \cite{Farb}.

In case the collection $\mc{H}$ is \emph{quasi-dense}, in the sense that every point is within bounded distance of some element of $\mc{H}$, there is another possibility, explored by Bowditch in~\cite{Bowditch12}.  Namely, one can let $\Gamma_{\mc{H},R}$ be a graph with vertex set $\mc{H}$, and connect two vertices if the corresponding sets have distance at most $R$.  For suitably chosen $R$, the space $\Gamma_{\mc{H},R}$ will be quasi-isometric to $\hat{X}$. 

We will use work of Kapovich--Rafi, Bowditch, and Dowdall--Taylor to understand $\partial \hat X$.  We start with a theorem of Bowditch (in the case $\mc{H}$ is quasi-dense) and Kapovich--Rafi.
\begin{theorem}[{\cite[Proposition~2.6]{KapovichRafi}, \cite[7.12]{Bowditch12}}]\label{thm:conehyp}
  Let $X$ be a Gromov hyperbolic graph, and let $\mc{H}$ be a  collection of uniformly quasi-convex subgraphs.  Then the cone-off of $X$ with respect to $\mc{H}$ is Gromov hyperbolic.  Moreover, geodesics in $X$ can be reparametrized to be uniform quasi-geodesics in $\hat{X}$.
  \end{theorem}
Combining Theorem \ref{thm:conehyp} with  quasi-geodesic stability (Theorem \ref{thm:stability}) yields the following.
\begin{corollary}\label{cor:nospiraling}
  Let $X$, $\mc{H}$, $\hat{X}$ be as above, and suppose that $\gamma$ is a quasi-geodesic ray in $X$.  Then the image of $\gamma$ in $\hat{X}$ is either bounded or has a unique limit point in $\partial \hat X$.
\end{corollary}

Since the cone-off is hyperbolic, it is natural to ask about its Gromov boundary.
\begin{question}
  If $\hat{X}$ is a cone-off of $X$, how is the Gromov boundary of $\hat{X}$ related to that of $X$?
\end{question}
Note that neither space quasi-isometrically embeds in the other, so it is not immediately clear that there should be any relation.  Dowdall and Taylor provide the following definition and theorem \cite{DowdallTaylor}.
\begin{definition}
  Let $\partial_u X\subseteq \partial X$ consist of the points represented by quasi-geodesic rays whose projection to $\hat X$ is unbounded.\footnote{Dowdall and Taylor use the notation $\partial_{\hat X} X$ in \cite{DowdallTaylor}.}
\end{definition}
Theorem~\ref{thm:conehyp} implies that the projection of a quasi-geodesic ray is unbounded if and only if it is a quasi-geodesic ray (after possibly reparametrizing).
\begin{theorem}[{\cite[Theorem 3.2]{DowdallTaylor}}]\label{thm:DTthm}
    The map $X\to \hat X$ extends continuously to $\partial_{u} X$, and the restriction $\psi\from \partial_{u} X \to \partial \hat X$ is a homeomorphism.
\end{theorem}

\begin{remark}
  Dowdall--Taylor's result is in the slightly more general setting of coarsely surjective \emph{alignment preserving} maps.  It follows from the ``Moreover'' statement in Kapovich--Rafi's Theorem~\ref{thm:conehyp} that the cone-off map $X\to \hat X$ is alignment preserving.
  
  Let $\Isom(X,\mc{H})$ be the group of isometries of $X$ which preserve the collection $\mc{H}$.  The homeomorphism $\psi$ from Theorem \ref{thm:DTthm} is $\Isom(X,\mc{H})$--equivariant.
\end{remark}

Recall that given $Y\in\mc H$, we denote by $\Lambda(Y)\subset \partial X$  the limit set of $Y$ in $\partial X$.  It is clear that $\Lambda(Y)$ is in the complement of $\partial_{u}X$, so
\begin{equation}\label{eq:include} \tag{\ddag}
  \partial_{u} X \subseteq \partial X \setminus \left(\bigcup_{Y\in\mc{H}}\Lambda(Y)\right).
\end{equation}
In general this can be a proper inclusion, as the following example shows.
\begin{example}\label{ex:bs12}
  Let $X$ be the Bass-Serre tree for $BS(1,2) = \langle a,t \mid tat^{-1}=a^2 \rangle$, so that $\partial X$ is a Cantor set.  Then $t$ acts loxodromically on $X$ with axis $\gamma$.  ($BS(1,2)$ does not act acylindrically along this axis, since the infinite cyclic group generated by $a$ stabilizes a sub-ray.)  Let $\mc{H}$ be the collection of translates of $\gamma$.  The union $\bigcup_{Y\in\mc{H}}\Lambda(Y)$ is a countable subset of $\partial X$, so $\partial X\setminus \bigcup_{Y\in \mc{H}}\Lambda(Y)$ is nonempty.  However the cone-off $\hat{X}$ of $X$ with respect to $\mc{H}$ is bounded, so $\partial \hat{X} = \emptyset$.  
\end{example}

We will show that if $(G,X,H)$ is an A/QI triple, then the inclusion \eqref{eq:include} is an equality.

\subsection{De-electrification}We need to be able to ``lift'' paths from $\hat{X}$ to $X$ in a consistent way.  There are several notions of this in the literature, usually called ``de-electrification.''   The following definition is from \cite{Spriano}, but it is related to definitions in Bowditch \cite{Bowditch12} and Dahmani--Mj \cite{DahmaniMj}.  What we call \emph{de-electrification} here is what Spriano calls \emph{embedded de-electrification}.

\begin{definition}\label{def:deelectrification}
  Let $\hat X$ be a cone-off of a graph $X$ with respect to a family of uniformly quasi-isometrically embedded subgraphs $\mc H$.  An \emph{$\mc{H}$--geodesic} is a path in $X$ which is completely contained in some $Y\in \mc{H}$ and is geodesic in $Y$.
  Let $\gamma=u_1\cdot e_1\cdot \ldots\cdot e_n\cdot u_{n+1}$ be a concatenation of geodesics in $\hat X$, where each $e_i$ is an electric edge and the $u_i$ are (possibly trivial) segments of $X$.  A \emph{de-electrification} of $\gamma$ is a concatenation $u_1\cdot \eta_1\cdot \ldots\cdot\eta_n\cdot u_{n+1}$ where each $\eta_i$ is an $\mc{H}$--geodesic joining the endpoints of $e_i$.  
\end{definition}

If the subgraphs in $\mathcal H$ are not geometrically separated, then a de-electrification of a quasi-geodesic in $\hat X$ can be arbitrarily far from being geodesic in $X$.  In particular we cannot expect every de-electrification to be quasi-geodesic.  However, Spriano shows that points of $\hat X$ are always connected by quasi-geodesics with nice de-electrifications.  Say that $K$ is a \emph{constant of quasi-isometric embeddedness} for $\mc{H}$ if every element of $\mc{H}$ is $(K,K)$--quasi-isometrically embedded, with $K$--quasi-convex image.

\begin{lemma}[{\cite[Corollary 2.29]{Spriano}}]\label{lem:embedded} 
  For any $\delta\ge 0$, $K\ge 1$, there is a constant $\omega\ge 1$ so that the following holds.
  Let $X$ be a $\delta$--hyperbolic graph and $\mc H$ a family of uniformly quasi-isometrically embedded subgraphs with constant of quasi-isometric embeddedness $K$.  Let $\hat X$ be the cone-off of $X$ with respect to $\mc H$.    Then for each pair of vertices $x,y\in X$ there exists an $\omega$--quasi-geodesic $\gamma$ of $\hat X$ from $x$ to $y$ so that every de-electrification of $\gamma$ is an $\omega$--quasi-geodesic.
  \end{lemma}

\putinbox{We fix the following data for the rest of the section:  $(G,X,H)$ is an A/QI triple where $X$ is a $\delta$--hyperbolic graph containing a basepoint $x_0$ and $Y_0$ is an $H$--cocompact graph in $X$ containing $x_0$ with constant of quasi-isometric embeddedness $K$.  The space $\hat X$ is the cone-off of $X$ with respect to the collection $\mc{H}$ of $G$--translates of $Y_0$.  The constant $\omega$ is obtained by applying Lemma~\ref{lem:embedded} to $\delta$ and $K$.}
\subsection{Characterizing rays with bounded image and proof of Theorem~\ref{thm:aqiboundary}}

\begin{lemma} \label{lem:bddimpliesY}
  Suppose $\gamma$ is a quasi-geodesic ray in $X$.  If $\gamma$ is bounded in $\hat X$ then $[\gamma]\in \bigcup_{g\in G}\Lambda(gHg^{-1})$.
 \end{lemma}

 \begin{proof}
   Fix $\tau_0$ so that $\gamma$ is a $\tau_0$--quasi-geodesic ray.  We can suppose that $\gamma$ begins at $x_0$.  Choose a sequence $\{x_i\}$ on $\gamma$ tending to infinity in $X$.  We will pass to subsequences several times but use the same notation for the subsequences.

   For each $i$, let $\sigma_i$ be the $\omega$--quasi-geodesic from $x_0$ to $x_i$ in $\hat X$ provided by Lemma~\ref{lem:embedded}, and let $\tilde\sigma_i$ be any de-electrification.  Then $\tilde\sigma_i$ is contained in an $M$--neighborhood of $\gamma$, where $M$ is the Morse constant for parameters $\max\{\tau_0,\omega\}$, $\delta$.
  \begin{figure}[htbp]
    \centering
    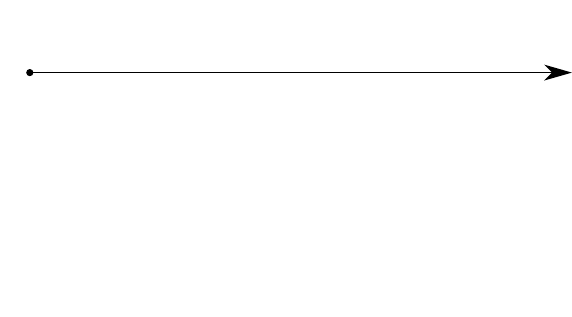\label{fig:spiral}
    \caption{De-electrifications of the paths $\sigma_i$ must have some subsegments whose lengths go to infinity.}
  \end{figure}

  Since $\gamma$ is bounded in $\hat X$, there is some $C\ge 0$ so that $\ell_{\hat X}(\sigma_i)\le C$ for all $i$.
  The quasi-geodesics $\tilde \sigma_i$  are piecewise quasi-geodesics of the form $\tilde\sigma_i=u_{i,1}\cdot \eta_{i,1}\cdot \ldots \cdot \eta_{i,k_i}\cdot u_{i,k_i+1}$, where the $u_{i,j}$ are (possibly trivial) paths in $X$, each $\eta_{i,j}$ is a geodesic in some $Y\in \mc{H}$, and $k_i\leq C$ for all $i$.  Since $\ell_X(\tilde\sigma_i)\to\infty$, some of the $\eta_{i,j}$ must be unbounded.

  Passing to a subseqence, we may assume that for a fixed $j$, the lengths $\ell_X(\eta_{i,j})$ tend to infinity, but the lengths of the prefixes $u_{i,1}\cdot \eta_{i,1}\ldots \cdot u_{i,j}$ are uniformly bounded by a constant $L$.  Set $\varepsilon = 3M+2L$, and let $D=D(\varepsilon,Y_0)$, $N=N(\varepsilon,Y_0)$ be the  constants of acylindricity of the action.  Discarding the first few terms of the sequence, we may assume that the $X$--distance between the endpoints of $\eta_{i,j}$ exceeds $D+2M$ for all $i$.  
Let $\gamma_{i,j}$ be the maximal subpath of $\gamma$ which is contained in the $M$--neighborhood of $\eta_{i,j}$.  Then the length of $\gamma_{i,j}$ is at least $D$.  Moreover, for  all $i$, the initial points of $\gamma_{i,j}$ are at distance at most $2L+2M$.  Let $y$ be a point on $\gamma$ closest to $x_0$ that is an endpoint of  $\gamma_{i,j}$ for some $i$, and let $\gamma'$ be the subpath of $\gamma$ beginning at $y$ having length $D$.

 Let $\mathcal Y\subseteq \mathcal H$ be the subset of $\mathcal H$ such that each $\eta_{i,j}\subset Y$ for some $Y\in \mathcal Y$.  Then $\gamma'\subset \mathcal N_{\varepsilon}(Y)$ for each $Y\in\mathcal Y$.

  Therefore we have 
  \[
  \diam\left(\bigcap_{Y\in\mathcal Y}\mathcal N_{\varepsilon}(Y)\right)\geq D.
  \]
As $\mathcal Y$ is a collection of translates of the $H$--cocompact subset $Y_0$, Theorem \ref{thm:newacyl} implies that $\mathcal Y$ can contain at most $N$ distinct elements. In particular the paths $\eta_{i,j}$ lie in a finite union $\bigcup_{k=1}^NY_k$ of elements of $\mc{H}$.  The paths $\eta_{i,j}$ fellow travel $\gamma$ for longer and longer intervals, so the ray $\gamma$ must limit to a point in $\Lambda(\bigcup_{k=1}^NY_k)$.  The limit set of a finite collection of quasi-convex sets is the union of their limit sets, so $[\gamma]\in \Lambda(Y_k)$ for some fixed $Y_k\in \mc{H}$.
\end{proof}

For the readers convenience we restate the main theorem of this section.
\aqiboundary*
\begin{proof}
  Dowdall--Taylor's Theorem \ref{thm:DTthm} tells us that $\partial \hat X$ is homeomorphic to the subspace $\partial_{u}X$ of $\partial X$ consisting of points represented by quasi-geodesic rays whose projection to $\hat X$ is unbounded.  But Lemma~\ref{lem:bddimpliesY} tells us that the only way for the image of a quasi-geodesic ray to be bounded is for it to limit to a point in $\Lambda(gHg^{-1})$ for some $g\in G$.  
\end{proof}

\subsection{Corollaries}

\begin{corollary}\label{cor:qiembeddedpersists}
Let $(G,X,H)$ be an A/QI triple,  let $K\leq G$ be quasi-isometrically embedded by the action of $G$ on $X$, and suppose $\Lambda(K)\subseteq \partial X-\bigcup_{g\in G}\Lambda(gHg^{-1})$.  Then $K$ is quasi-isometrically embedded by the action of $G$ on $\hat X$.  
\end{corollary}

The proof of this corollary will make use of a version of the Milnor--Schwartz Lemma which applies to quasi-geodesic spaces.  A metric space is \emph{quasi-geodesic} if there exists some $\tau\ge 1$ so that any two points are joined by a $\tau$--quasi-geodesic.  

\begin{lemma}[{\cite[Proposition~5.4.1]{Loh}}]
If a group $G$ acts coboundedly and metrically properly on a quasi-geodesic metric space $Z$, then $G$ is finitely generated and for any $z\in Z$, the orbit map $g\mapsto gz$ is a quasi-isometry.
\end{lemma}

\begin{proof}[Proof of Corollary \ref{cor:qiembeddedpersists}]
  We will apply the Milnor--Schwartz Lemma with $Z=Kx_0\subseteq \hat X$, from which the result is immediate.   It is clear that $K$ acts cobounded on $Kx_0$, so it remains to show that $(Kx_0, d_{\hat{X}})$ is quasi-geodesic and that $K$ acts metrically properly.

  By Theorem~\ref{thm:conehyp}, geodesics in $X$ can be reparametrized to be $\tau$--quasi-geodesics in $\hat{X}$ for some fixed $\tau\ge 1$.  Let $M$ be the Morse constant for $\tau$--quasi-geodesics in $\hat{X}$, and let $\lambda$ be the quasi-convexity constant of $Kx_0$ in $X$.  
  
We first show that $(Kx_0, d_{\hat{X}})$ is a quasi-geodesic metric space.  As the reader may verify, it suffices to show that $Kx_0$ is quasi-convex in $\hat{X}$.  Let $x,y$ be points of $Kx_0$, and let $\beta$ be a geodesic of $\hat{X}$ joining them.  Let $z\in \beta$, and let $\alpha$ be an $X$--geodesic joining $x$ to $y$.  Since $\alpha$ is (after reparametrizing) a $\tau$--quasi-geodesic in $\hat{X}$, there is a point $z'$ of $\alpha$ so that $d_{\hat{X}}(z,z')\le M$.  Since $\alpha$ is a geodesic in $X$, there is a point $z''\in K x_0$ so that $d_X(z',z'')\le \lambda$.  So
\[d_{\hat{X}}(z,K x_0)\le d_{\hat{X}}(z,z') + d_{\hat{X}}(z',z'') \le d_{\hat{X}}(z,z') + d_X(z',z'') \le M + \lambda,\]
and so $K x_0$ is $(M+\lambda)$--quasi-convex in $\hat{X}$.

To show that $K$ acts metrically properly on $Kx_0\subseteq \hat X$, let $B\subset Kx_0$ be a finite diameter subset, and suppose toward a contradiction that there exists an infinite collection of elements $k_i\in K$ such that $k_i B\cap B\neq \emptyset$.  Since $K$ is quasi-isometrically embedded by the action of $G$ on $X$, the points $k_ix_0$ do not stay in any bounded subset of $X$.
After passing to a subsequence, we may assume the points $k_ix_0\in X$ converge to a point  $\xi \in \Lambda(K)\subseteq \partial X$.   The set $\{k_ix_0\}$ is bounded in $\hat X$, so the point $\xi$ cannot persist in $\partial \hat X$.  However, this contradicts Theorem~\ref{thm:aqiboundary}, as  $\Lambda(K)\subseteq \partial X-\bigcup_{g\in G}\Lambda(gH)$.
\end{proof}

We now use the description of the boundary of $\hat X$ given in Theorem~\ref{thm:aqiboundary} to understand the action of $G$ on $\hat X$. 

If a group $G$ acts by isometries on a Gromov hyperbolic metric space $Y$ with basepoint $y_0$, then  $g\in G$ is \emph{elliptic} if some (equivalently any) orbit in $Y$ is bounded, \emph{loxodromic} if the map $\mathbb Z\to Y$ defined by $n\mapsto g^ny_0$ is a quasi-isometric embedding, and \emph{parabolic} if it is neither elliptic nor loxodromic.  One may also use limit sets in $\partial Y$ to distinguish between these types of isometries: an element $g\in G$ is elliptic, parabolic, or loxodromic if $\#\Lambda(\langle g\rangle)=0,1,$ or $2$, respectively.  If $\#\Lambda(\langle g\rangle)=2$, then we let $\Lambda(\langle g\rangle)=\{g^{\pm \infty}\}$.  In this case, there is a quasi-geodesic axis in $Y$ limiting to $g^{\pm\infty}$, and $g$ acts as translation along this axis.

As $X$ is a hyperbolic space, every element of $G$ is either elliptic, loxodromic, or parabolic with respect to $G\curvearrowright X$.  Since the canonical map from $X$ to $\hat X$ is 1--Lipschitz and $G$--equivariant, if $g\in G$ is elliptic with respect to $G\curvearrowright X$, then $g$ is also elliptic with respect to $G\curvearrowright \hat X$.  It remains to consider elements of $G$ that are parabolic and loxodromic with respect to $G\curvearrowright X$.  

\begin{corollary}  \label{cor:loxell}
  Suppose $(G,X,H)$ is an A/QI triple, and let $\mathcal{H}$, $\hat{X}$ be as in Theorem~\ref{thm:aqiboundary}.
  \begin{enumerate}
  \item\label{item:loxodromic} If $g\in G$ acts loxodromically on $X$, then either it acts loxodromically on $\hat{X}$ or a power of $g$ stabilizes some $Y\in \mathcal{H}$, and so $g$ acts elliptically on $\hat{X}$.
  \item\label{item:parabolic} An element $g\in G$ acts parabolically on $X$ if and only if it acts parabolically on $\hat{X}$.
  \end{enumerate}
\end{corollary}
\begin{proof}
  Suppose that $g$ acts loxodromically on $X$.  Then $\langle g\rangle$ is quasi-isometrically embedded by the action of $G$ on $X$.  If each of two fixed points of $g$ in $\partial X$ avoids $\Lambda(Y)$ for all $Y\in\mc H$, then $\langle g\rangle$ is quasi-isometrically embedded by the action of $G$ on $\hat X$ by Corollary \ref{cor:qiembeddedpersists}, hence $g$ acts loxodromically on $\hat X$.

Thus we may assume that at least one of the fixed points of $g$ lies in the limit set of an element $Y\in \mc H$.   Up to  replacing $g$ by $g^{-1}$, we can assume that $\Lambda(Y)$ contains the repelling fixed point $g^{-\infty}\in \partial X$.  The following claim implies statement~\eqref{item:loxodromic} of the corollary.
  \begin{claim}
    \label{claim:lox to ell}  Some power of $g$ stabilizes $Y$.
  \end{claim}
  \begin{proof}
    Let $\alpha_g\from \bR\to X$ be a quasi-geodesic axis for $g$.
    Let $y_0\in Y$ be a closest vertex to $\alpha_g(0)$, and let $\gamma_0$ be a quasi-geodesic ray lying in $Y$ and limiting to $g^{-\infty}$.  By quasi-geodesic stability, there is some $B>0$ so that $d_X(\alpha_g(0),y_0)\le B$ and the quasi-geodesic rays $\gamma_0$ and $\alpha_g|_{(-\infty,0]}$ are Hausdorff distance at most $B$ from one another.

Fix $\varepsilon=B$, and let $D=D(\varepsilon,Y)$ and $N=N(\varepsilon,Y)$ be the constants  of acylindricity of the action.

    For any integer $i\ge 0$, there is a $T_i\ge 0$ so that $\alpha_g(T_i) = g^i\alpha_g(0)$, and we may assume the $T_i$ tend monotonically to infinity.  By equivariance, $\alpha_g|_{(-\infty,T_i]}$ is Hausdorff distance at most $B$ from $\gamma_i = g^i\gamma_0 \subseteq g^i Y$.  All these rays tend to $g^{-\infty}$ and lie in the $B$--neighborhood of $\alpha_g$.
    
    In particular, $\alpha_g|_{(-\infty,T_0]}\subset \mathcal N_B(\gamma_i)\subseteq \mathcal N_B(g^iY)$ for each $i\geq 0$.  Therefore,
    \[
    \diam\left(\bigcap_{i=1}^\infty\mathcal N_B(g^iY)\right)>D.
    \]
    By Theorem \ref{thm:newacyl}, the collection $\{g^iY\}$ contains at most $N$ distinct elements.  

This implies that there is some $j'\neq j$ such that $g^{j'} Y=g^{j}Y$.  Therefore, $g^{j'-j}$ stabilizes the subgraph $Y$.
  \end{proof}

  We next deal with parabolic elements.  If $g\in G$ acts parabolically on $\hat{X}$, then it has unbounded orbits in $X$, so it cannot act elliptically on $X$.  But it also cannot act loxodromically, by the first statement.  The only remaining possibility is that it acts parabolically.

Conversely, suppose that $g$ acts parabolically on $X$.  Then $g$ fixes a single point $p\in \partial X$.
\begin{claim} \label{claim:fp in imphi}
  The fixed point $p$ of $g$ lies in $\partial_{u}X$.
\end{claim}
\begin{proof}
  We argue by contradiction, supposing that $p$ is not in $\partial_{u}X$.  By Theorem~\ref{thm:aqiboundary}, this means that $p\in \Lambda(Y)$ for some $Y\in \mc{H}$.  The proof proceeds similarly to the proof of Claim \ref{claim:lox to ell} with the following changes:  we no longer have a quasi-axes $\alpha_g$, but we instead simply choose some $y_0\in Y$ and consider a $K$--quasi-geodesic ray $\gamma_0$ in $Y$ from $y_0$ to $p$.  For some infinite subset $J\subseteq \mathbb Z_{\ge 0}$ and some fixed $B>0$ as before, there are numbers $T_j$ tending monotonically to infinity so that for all $t\geq T_j$, $g^j\gamma_0(t)=\gamma_j(t)$ is contained in the $B$--neighborhood of $\gamma_0$.  
  
  Arguing as in the proof of Claim \ref{claim:lox to ell}, we conclude that there is some $j'\neq j$ such that $g^{j'-j}$ stabilizes the subgraph $Y$.  However, since $Y$ is a uniformly locally finite graph, it cannot be stabilized by a parabolic isometry.  Thus we reach a contradiction, and we conclude that $p$ lies in $\partial_{u}X$.
\end{proof}

Let $q=h(p)$ for some loxodromic element $h\in H$, and note that $q$ is in $\partial_{u}X$.  The points $\{g^i q\mid i\in \mathbb Z\}$ accumulate on $p$ in $\partial \hat X$.  If $g$ had a bounded orbit in $\hat{X}$, this would be impossible. (To see this, consider the image of a $\tau_0$--quasi-geodesic joining $p$ to $q$.) Thus $g$ acts parabolically on $\hat X$.
\end{proof}

There are several similar results in the literature for cone-offs, with varying conditions on $X$ and the subgroup $H$, for example \cite[Theorem~1.14]{Osin06} and \cite[Proposition~6.5]{ABO}, which both conclude that if $g$ is loxodromic on $X$ then it is either loxodromic on the cone-off or conjugate into $H$.  In both of these results, there is some hypothesis of properness, which allows the slightly stronger conclusion.  

We now consider actions $G\curvearrowright X$ which contain loxodromic WPD elements (see Definition~\ref{def:WPD}).  We show that if $g$ is WPD with respect to $X$, and if it is still loxodromic for the action on the cone-off $\hat X$, then it is WPD with respect to $\hat X$.  The proof will use the notion of a WWPD element, introduced by Bestvina--Bromberg--Fujiwara in \cite{BestvinaBrombergFujiwaraSCL}.  The definition we use here is due to Handel--Mosher in \cite{HandelMosher:WWPD}: a loxodromic element $g\in G$ is \emph{WWPD} if the  $G$--orbit of $(g^\infty,g^{-\infty})$ is a discrete subset of $\partial X\times\partial X\setminus \Delta$, where $\Delta$ is the diagonal subset.  A loxodromic element $g$ is WPD if and only if it is WWPD and the stabilizer of $\{g^\infty,g^{-\infty}\}$ in $G$ is virtually cyclic \cite[Corollary~2.4]{HandelMosher:WWPD}.

 The following proposition is related to~\cite[Theorem~2.4]{MMS20}, but has slightly different hypotheses.
\begin{proposition}\label{prop:WPDpersists}
  Let $(G,X,H)$ be an A/QI triple.  If $g\in G$ is a loxodromic WPD element with respect to the action of $G$ on $X$ and no power of $g$ stabilizes any $Y\in \mc{H}$, then $g$ is a loxodromic WPD element with respect to the action on $\hat X$.
\end{proposition}

\begin{proof}
  By Corollary~\ref{cor:loxell}, the element $g$ acts loxodromically on $\hat X$ and the fixed points of $g$ in $\partial X$ lie in $\partial_{u}X$.  Since $g$ is WWPD, the $G$--orbit of $(g^\infty,g^{-\infty})$ is a discrete subset of $\partial X\times \partial X\setminus \Delta$.  This orbit is therefore discrete in $\partial \hat X\times \partial \hat X\setminus \Delta\cong \partial_{u}X\times\partial_{u} X\setminus \Delta$.  Therefore $g$ is a WWPD element with respect to the action of $G$ on $\hat X$.  Since $g$ is WPD on $X$,  the stabilizer of $\{g^\infty,g^{-\infty}\}$ with respect to the action of $G$ on $\partial X$ is virtually cyclic.  Moreover, the stabilizer of the pair $\{g^{\infty},g^{-\infty}\}$ does not change when we consider the action of $G$ on $\partial \hat X\cong \partial X - \bigcup_{g\in G}\Lambda(gH)$.  Therefore $g$ is WPD with respect to the action on $\hat X$, as desired.
\end{proof}

\begin{corollary} \label{cor:wpd}
  Let $G, X, \mathcal{H},$ and $\hat X$ be as in Theorem~\ref{thm:aqiboundary}, and suppose additionally that $G\curvearrowright X$ is a WPD action.  If there is a loxodromic isometry in the action of $G$ on $\hat X$, then $G\curvearrowright\hat X$ is a WPD action.
\end{corollary}

\section{Examples}\label{sec:examples}

In this section, we use the combination theorem presented in Section \ref{sec:finiteheight} to give examples of A/QI triples in $\OutFn$, the Cremona group, and FC-type Artin groups (see Examples \ref{ex:outfn}, \ref{ex:cremona}, and \ref{ex:artin}).   All of these groups are acylindrically hyperbolic, but while we have natural hyperbolic spaces on which each group acts with loxodromic WPD elements, these actions are not known to be acylindrical (and in several cases are known not to be acylindrical). Constructing A/QI triples in these situations provides subspaces along which the action of the group is  acylindrical.  

We first consider the group $\OutFn$, which acts by isometries on several hyperbolic complexes, including the free factor complex, the cyclic splitting complex, and the free splitting complex.  We refer the reader to \cite{HandelMosher:FreeSplitting, KapovichRafi, Mann} for details on these complexes and the actions.  Fully irreducible elements are loxodromic WPD elements with respect to the action on all three complexes.  These are the only loxodromic isometries of the action on the free factor complex, and so this is a WPD action.  However, it is not known if this action is acylindrical.  It is also not known if the action on the cyclic splitting complex is WPD or acylindrical. On the other hand, Handel and Mosher \cite{HandelMosher:FreeSplitting} show that there are loxodromic isometries of the free splitting complex which are not WPD elements, and so this action is neither WPD nor acylindrical. However, in all cases, we are able to construct A/QI triples using Corollary~\ref{prop:WPDtoAQI}.  As noted in Remark~\ref{rem:notHE}, the subgroups  constructed will typically \emph{not} be hyperbolically embedded.

\begin{example}[$\OutFn$] \label{ex:outfn}
Let $G=\OutFn$,  let $h_1,\dots, h_k\in G$ be independent fully irreducible elements, and let $X$ be the free factor complex, the cyclic splitting complex, or the free splitting complex.  Then $h_1,\dots, h_k$ are independent loxodromic WPD elements with respect to the action on $X$.  By Corollary~\ref{prop:WPDtoAQI}, there is a constant $D$ such that for any $d_1,\dots, d_m\geq D$, if $H=\langle h_1^{d_1},\dots, h_k^{d_k}\rangle$, then $(G,X,H)$ is an A/QI triple. 
\end{example}

A subgroup of $\OutFn$ is called \emph{convex cocompact} if it is quasi-isometrically embedded in the free factor complex via the orbit map.  Thus all the subgroups in Example \ref{ex:outfn} are convex cocompact.  These subgroups are known to be stable by \cite{AougabDurhamTaylor}.

We now turn our attention to the Cremona group.
 
\begin{example}[Cremona group] \label{ex:cremona}
The $2$--dimensional Cremona group $\operatorname{Bir}(\mathbb P^2_k)$ over an algebraically closed field $k$ is the group of birational transformations of the projective space $\mathbb P^2_k$.  We note that $\operatorname{Bir}(\mathbb P^2_k)$ is not a finitely generated group; in fact, it is uncountable.  In \cite{CantatLamy}, Cantat and Lamy construct an action of $\operatorname{Bir}(\mathbb P^2_k)$  on an 
 infinite-dimensional hyperbolic space $\mathbb H_{\overline{\mathcal Z}}$, which is inspired by work of Manin \cite{Manin} and Zariski \cite{Zariski}.    Cantat and Lamy also introduce the notion of a \emph{tight} loxodromic isometry of $\mathbb H_{\overline{\mathcal Z}}$; see \cite[Section~2.3.3]{CantatLamy}.  An element $g\in \operatorname{Bir}(\mathbb P^2_k)$ is  WPD if and only if some positive iterate $g^m$ of $g$ is a tight loxodromic isometry \cite[Remark~7.5]{Cantat}.  (Lonjou also constructs WPD elements in the case that $k$ is not algebraically closed \cite{Lonjou}.)

Let $G=\operatorname{Bir}(\mathbb P^2_k)$, and let $h_1,\dots, h_k\in G$ be independent tight elements.  Then $h_1,\dots,h_n$ are independent loxodromic WPD isometries for the action on  $\mathbb H_{\overline{\mathcal Z}}$. By Corollary \ref{prop:WPDtoAQI}, there is a constant $D$ such that for any $d_1,\dots, d_n\geq D$ and $H=\langle h_1^{d_1},\dots, h_n^{d_n}\rangle$, we have that $(G,\mathbb H_{\overline{\mathcal Z}},H)$ is an A/QI triple.  Since $G$ is not finitely generated, there is not a notion of a stable subgroup, but Corollary \ref{prop:finiteheight} implies that $H$ has finite height.
\end{example}

For our final example, which is more involved than the previous examples, we consider FC-type Artin groups. 
\begin{example}[FC-type Artin groups]\label{ex:artin}
 Let $\Gamma$ be a finite graph with vertex set $V(\Gamma)$ and edge set $E(\Gamma)\subset V(\Gamma)\times V(\Gamma)$.  Suppose that each edge $(s,t)\in E(\Gamma)$ is labeled with an integer $m(s,t)\ge 2$.  The \emph{Artin group} $A(\Gamma)$ associated to the labeled graph $\Gamma$ is defined by
\[
A_\Gamma=\left\langle V(\Gamma) \mid \underbrace{st\cdots}_{m(s,t)}=\underbrace{ts\cdots}_{m(s,t)} \textrm{ for all } (s,t)\in E(\Gamma) \right\rangle.
\]
Let $S=V(\Gamma)$.  Given a collection of vertices $T\subset S$, the subgroup $A_T$ of $A_\Gamma$ generated by $T$ is isomorphic to the Artin group associated to the subgraph of $\Gamma$ spanned by $T$.  An Artin group is \emph{spherical} if the associated Coxeter group (which is formed by adding the relation $s^2=1$ for all $s\in S$) is finite.  An important larger class of Artin group are the FC-type Artin groups, which were first introduced by Charney--Davis \cite{CharneyDavis}.  An Artin group is \emph{FC-type} if $A_T$ is spherical if and only if $T$ is a complete subgraph.

An Artin group $A_\Gamma$ acts by isometries on its \emph{Deligne complex} $\mathcal D(A_\Gamma)$.  The Deligne complex (called the \emph{modified} Deligne complex by Charney--Davis) is a cubical complex whose vertices  are cosets $aA_T$ where $A_T$ is spherical.  Whenever $T\subset T'$ and $\#(T'-T) = 1$, an edge is added between the cosets $aA_T$ and $aA_{T'}$; higher-dimensional cubes are filled in as their $1$--skeleta appear.  
  Charney--Davis show in \cite{CharneyDavis} that the Deligne complex $\mathcal D(A_\Gamma)$ is CAT$(0)$ exactly when  $A_\Gamma$ is FC-type.

Chatterji--Martin show that FC-type Artin groups are acylindrically hyperbolic  by applying the following general result  about groups acting on CAT$(0)$ cube complexes, whose proof relies on \cite[Theorem~1.2]{Martin2}.  We refer the reader to \cite{ChatterjiMartin} for the definitions of the terms appearing in theorem. The action of $A_\Gamma\curvearrowright \mathcal D(A_\Gamma)$ satisfies the conditions in the first sentence by \cite[Proposition~5.1]{ChatterjiMartin}.

\begin{theorem}[{\cite[Theorem~1.1]{ChatterjiMartin}}]\label{thm:CM1.1}
Let $G$ be a group acting essentially and non-elementarily on an irreducible finite-dimensional CAT$(0)$ cube complex.  If there exist two hyperplanes whose stabilizers intersect along a finite subgroup, then $G$ admits an acylindrical action on a hyperbolic space.
\end{theorem} 

The proof of this theorem has two steps.  The first is to show that there exists a loxodromic WPD isometry $g\in G$ with respect to the action on the CAT$(0)$ cube complex whose axis in $G$ is strongly contracting.  Since the CAT$(0)$ cube complex is not required to be hyperbolic, this is not sufficient to prove that $G$ is acylindrically hyperbolic. The second step is to  apply the projection complex machinery of Bestvina--Bromberg--Fujiwara from \cite{BestvinaBrombergFujiwara} to conclude that $g$ is a loxodromic WPD  isometry in an action of $G$ on a hyperbolic space.  Osin shows that this is equivalent to $G$ acting acylindrically on a (possibly different) hyperbolic space \cite[Theorem~1.4]{Osin16}.  The projection complex is not the only hyperbolic space that could be used as the input for Osin's theorem.  Genevois constructs a family of hyperbolic spaces on which $G$ acts by isometries, and the element $g$ will be a loxodromic WPD isometry of infinitely many of these spaces \cite{Genevois}.  (Genevois's construction generalizes Hagen's construction of the contact graph \cite{Hagen}.)

 A graph \emph{has no empty squares} if every circuit of length four has at least one diagonal pair of vertices spanning an edge.  Charney--Crisp \cite[Theorem~4.2]{CharneyCrisp} show that if $A_\Gamma$ is an FC-type Artin group and  $\Gamma$ has no empty squares, then $\mathcal D(A_\Gamma)$ is actually CAT$(-1)$; in particular, it is hyperbolic.  
Thus to prove Theorem \ref{thm:CM1.1} in the special case of an FC-type Artin group whose defining graph has no empty squares, the Deligne complex can be used directly as the input to Osin's theorem.  

In this setting, the following statement can be extracted from the proof of \cite[Theorem~1.1]{ChatterjiMartin}. Given disjoint hyperplanes $\hat H_1,\hat H_2$ in an irreducible CAT$(0)$ cube complex $X$, an automorphism $g$ of $X$  \emph{double-skewers} $\hat H_1$ and $\hat H_2$ if there is a  choice of corresponding nested half spaces $H_1\subset H_2$ such that $H_1\subset H_2\subset gH_1$.
\begin{proposition}[{\cite[Theorem~1.1]{ChatterjiMartin}}] Let $A_\Gamma$ be an Artin group of FC type whose defining graph has no empty squares.  If there exist two disjoint hyperplanes in the Deligne complex $\mathcal D(A_\Gamma)$ whose stabilizers intersect along a finite subgroup, then any element which double-skewers these hyperplanes is a loxodromic WPD element with respect to the action of $A_\Gamma$ on  $\mathcal D(A_\Gamma)$.  
\end{proposition}
Chatterji--Martin use the Sector Lemma \cite[Lemma~5.2]{CapraceSageev} and the Double-Skewering Lemma \cite[Double-Skewering~Lemma]{CapraceSageev} to show that such elements must always exist. However, the proof does not explicitly construct such elements.  We now explicitly construct a family of examples.

The construction involves understanding the intersection pattern of hyperplanes in the Deligne complex.  As described in \cite[Remark~2.2]{MorrisWright},  every hyperplane in $\mathcal D(A_\Gamma)$ is a translate of a hyperplane that crosses an edge $(A_{\emptyset},A_{\{v\}})$ for some vertex $v$ in $\Gamma$.  Such a hyperplane is called a \emph{hyperplane of type $v$}.  The action of $A_\Gamma$ on the Deligne complex preserves hyperplane types, and a hyperplane of type $v$ intersects some hyperplane of type $w$ exactly when $w$ is in the link of $v$ in $\Gamma$.

\begin{lemma}
Let $A_\Gamma$ be an FC-type Artin group whose defining graph $\Gamma$ has diameter at least 3 and no empty squares.  For any two vertices $s,t\in V(\Gamma)$ at distance at least 3 in $\Gamma$ and any integers $m,n\in\mathbb Z\setminus\{0\}$, the product $s^mt^n$ is a loxodromic WPD element with respect to the action of $A_\Gamma$ on the Deligne complex.  Moreover, if $m',n'\in\mathbb Z\setminus\{0\}$ and $(m',n')\neq (m,n)$, then the elements $s^{m'}t^{n'}$ and $s^mt^n$ are independent.
\end{lemma}

\begin{proof}
If $\Gamma$ is not connected, then $A_\Gamma$ splits as a free product and the Deligne complex $\mathcal D(A_\Gamma)$ has the structure of a tree of spaces.  If $s,t$ are in different connected components of $\Gamma$, the result is clear.  Thus we may assume that $\Gamma$ is connected.

Let $\hat{H}_s$ be the hyperplane of $\mathcal D(A_\Gamma)$ which crosses the edge $(A_\emptyset,A_{\{s\}})$, and let $\hat{H}_t$ be the hyperplane which crosses the edge $(A_\emptyset,A_{\{t\}})$.
As noted in the proof of \cite[Theorem~1.2]{ChatterjiMartin}, since $s$ and $t$ are at distance at least 3 in $\Gamma$, we have $\Stab(\hat{H}_s)\cap\Stab(\hat{H}_t)=\emptyset$.  Moreover, $s^mt^n$ is a loxodromic isometry of $\mathcal D(A_\Gamma)$ which has a geodesic axis formed by the union of the images of the segment shown in Figure \ref{fig:Deligneaxis} under $\langle s^mt^n\rangle$ by \cite[Lemma~5.5]{ChatterjiMartin}. (Strictly speaking, \cite[Lemma~5.5]{ChatterjiMartin}  only applies to the element $st$, but the same proof will give the result for $s^mt^n$; this is explicitly stated for the element $s^2t$ in the proof of \cite[Lemma~5.4]{ChatterjiMartin}.)

\begin{figure}[htbp]
\centering
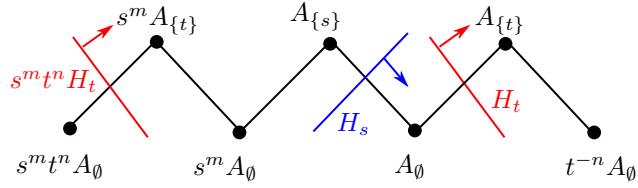 
\caption{A geodesic axis of $s^mt^n$ in $\mathcal D(A_\Gamma)$, with the relevant half spaces.}
\label{fig:Deligneaxis}
\end{figure}

Fix $m,n\in \mathbb Z\setminus\{0\}$.  Since $t$ is not contained in the link of $s$, we have $\hat H_s\cap\hat H_t=\emptyset$.  Thus there exist half spaces $H_s$ and $H_t$ satisfying $H_t\subset H_s$ (see Figure \ref{fig:Deligneaxis}).  We will show that $s^mt^n$ double skewers $ H_s$ and $ H_t$, that is, we will show that $H_t\subset H_s\subset s^mt^nH_t$.  First, notice that since  $s$ is not contained in the link of $t$, the hyperplanes $\hat H_s$ and $s^mt^n\hat H_t$  are disjoint. Thus $H_s\subset s^mt^n H_t$ or $s^mt^n H_t\subset H_s$.  If $s^mt^n H_t\subset H_s$, then $\hat H_s$ must intersect the path $A_{\{s\}}, s^mA_\emptyset, s^mA_{\{t\}}$ exactly once.  However, since $\hat H_s$ is of type $s$, it does not cross the edge $(s^mA_\emptyset, s^mA_{\{t\}})$.  Moreover, since hyperplanes in a CAT$(0)$ cube complex do not self-osculate and $\hat H_s$ crosses the edge $(A_\emptyset,A_{\{s\}})$, we see that $\hat H_s$  cannot cross $(A_{\{s\}}, s^mA_\emptyset)$.  Thus $H_s\subset s^mt^n H_t$, and we conclude that $s^mt^n$ double-skewers ${H_s}$ and ${H_t}$.

Therefore $s^mt^n$ is a loxodromic WPD element with respect to the action of $A_\Gamma$ on the Deligne complex by \cite[Theorem~1.1]{ChatterjiMartin}.

The proof of \cite[Lemma 5.4]{ChatterjiMartin} shows that $st$ and $s^2t$ are independent, but the proof goes through as written for elements $s^{m'}t^{n'}$ and $s^mt^n$ whenever $(m',n')\neq (m,n)$.
\end{proof}

Let $A_\Gamma$ be an FC-type Artin group whose defining graph $\Gamma$ has diameter at least 3 and no empty squares.  Let $s,t\in V(\Gamma)$ be at distance at least 3, and fix distinct elements $h_1,\dots, h_k\in\{s^mt^n\mid m,n\in \mathbb Z\setminus\{0\}\}$.  By Corollary \ref{prop:WPDtoAQI} there is a constant $D$ such that for any $d_1,\dots, d_k\geq D$ and $H=\langle h_1^{d_1},\dots,h_k^{d_k}\rangle$, we have that $(A_\Gamma, \mathcal D(A_\Gamma), H)$ is an A/QI triple.  In particular, the subgroup $H$ is stable.
\end{example}

\section{Questions}\label{sec:questions}

The first question we pose asks for a generalization of Proposition \ref{prop:intersectedtriples}.
\begin{question}
  Suppose $(G,X,H)$ and $(G,Y,K)$ are A/QI triples but $X$ and $Y$ are not equivariantly quasi-isometric.  Does there exist an action $G\acts Z$ such that $(G,Z,H\cap K)$ is an A/QI triple?  
\end{question}

The next question asks if there is a more general combination theorem than the one presented in Theorem \ref{thm:combination}. 
\begin{question}
Suppose $(G,X,H)$ and $(G,X,K)$ are  A/QI triples and $H_0\leq H$ and $K_0\leq K$ are quasi-convex subgroups such that every element of $H_0 \setminus K$ and every element of $K_0\setminus H$ translates the basepoint of $X$ a large distance. Let $P=\langle H_0,K_0\rangle$.  Is $(G,X,P)$ an A/QI triple? 
\end{question}

Corollary \ref{cor:wpd} shows that if $(G,X,H)$ is an A/QI triple and $G\curvearrowright X$ is WPD, then $G\curvearrowright \hat X$ is also WPD. 
Since an acylindrical action is a WPD action, we can conclude that if $G\curvearrowright X$ is acylindrical, then $G\curvearrowright \hat X$ is  WPD.  This motivates the following question.
\begin{question}
Suppose that $(G,X,H)$ is an A/QI triple where $G\acts X$ is actually acylindrical.  Form the cone-off $\hat{X}$ as in Section~\ref{sec:coneoff}.  Is the action $G\acts \hat{X}$ also acylindrical?
\end{question}

Example \ref{ex:outfn} shows that many known convex cocompact subgroups of $\OutFn$ give A/QI triples.  Moreover, every convex cocompact subgroup of $\OutFn$ quasi-isometrically embeds in the free factor complex. 
\begin{question}
Does $\OutFn$ always act acylindrically along the image of a convex cocompact subgroup in the free factor complex? 
\end{question}

As well as finite height, Antol\'in--Mj--Sisto--Taylor show finite width and bounded packing for stable subgroups \cite{AMST19}.  We would like to know if these properties hold for $H$ in an A/QI triple $(G,X,H)$ even when $G$ is not finitely generated.
\begin{question}
When $G$ is not finitely generated, but there is an A/QI triple $(G,X,H)$, does $H$ have finite width?  If $G$ is countable, does $H$ have bounded packing?
\end{question}

Part of the usefulness of finite height in context of quasi-convex subgroups of hyperbolic groups is that one can pass to an \emph{almost malnormal core} of a finite collection of quasi-convex groups.  This core is a collection of quasi-convex subgroups of height $1$, so the ambient group is hyperbolic relative to this collection, and one can attempt to run various inductive arguments on height using Dehn filling (as for example in \cite{AGM09} and the appendix to \cite{Agol13}).  Rather than aiming to produce a relatively hyperbolic structure on $G$, one might hope to produce a hyperbolically embedded collection of subgroups.  Such a collection would necessarily have height one.
As we have mentioned, the $H$ in an A/QI triple $(G,X,H)$ is not necessarily hyperbolically embedded as its height may be strictly bigger than one.  In order to pass to a well-behaved almost malnormal core, one would like to know the answer to the following question.
\begin{question}
  Consider the collection of infinite subgroups of $H$ of the form $\bigcap_{i=1}^{k} g_k H g_k^{-1}$.  Do these subgroups fall in finitely many $H$--conjugacy classes?
  Do the minimal such subgroups give rise to a hyperbolically embedded collection in $G$ (after possibly replacing each by a finite index supergroup and choosing one per $G$--conjugacy class)?
\end{question}

Using the machinery of hierarchically hyperbolic groups, Spriano in \cite{Spriano} proves Theorem~\ref{thm:aqiboundary} in the special case where $\mc H$ is a collection of infinite quasi-convex subspaces which can be extended to form a \emph{weak factor system}.  Indeed Spriano gives finer information about the boundary in this case, decomposing it into the boundaries of the various elements of the factor system, one of which is $\partial \hat X$.  When $X$ is the Cayley graph of a hyperbolic group and $\mc H$ is the set of cosets of a finite collection of infinite quasi-convex subgroups, Spriano shows that $\mc H$ extends to a weak factor system \cite[Section 5.2]{Spriano}.  Spriano's argument seems to strongly use properness.  In general, $X$ may not be a proper metric space, and thus we ask the following question:

\begin{question} Under the assumptions of Theorem~\ref{thm:aqiboundary}, can the set of translates $gHx$ can be extended to form a weak factor system?
\end{question}

\newcommand{\etalchar}[1]{$^{#1}$}

\end{document}